\documentclass[11pt]{amsart}

\usepackage[utf8]{inputenc}
\usepackage[backend=biber,style=alphabetic,maxbibnames=50]{biblatex}
\usepackage{amssymb,amsfonts}
\usepackage{amsmath,amscd}
\usepackage[all,arc]{xy}
\usepackage{enumerate}
\usepackage{mathrsfs}
\usepackage[pdftex]{graphicx}
\usepackage[T1]{fontenc}
\usepackage[usenames,dvipsnames]{color}
\usepackage[bookmarks=false]{hyperref}
\usepackage{dsfont}
\usepackage{tikz-cd}
\usetikzlibrary{automata,positioning}
\usepackage{fullpage}
\usepackage{wasysym}

\numberwithin{equation}{section}

\theoremstyle{plain}
\newtheorem{thm}{Theorem}[section]
\newtheorem{cor}[thm]{Corollary}
\newtheorem{prop}[thm]{Proposition}
\newtheorem{lem}[thm]{Lemma}

\theoremstyle{definition}
\newtheorem{defn}[thm]{Definition}

\newtheorem{aDD^+m}[thm]{ADD^+endum}

\theoremstyle{remark}
\newtheorem{rmk}[thm]{Remark}

\title{Joining rigidity for rational maps}

\author{Fabrizio Bianchi}
\address{Dipartimento di Matematica, Università di Pisa, Largo Bruno Pontecorvo 5, 56127 Pisa, Italy}
 \email{fabrizio.bianchi$@$unipi.it}
\author{Yan Mary He}
\address{Department of Mathematics\\
	University of Oklahoma\\
	Norman, OK 73019}
\email{he$@$ou.edu}
\date{\today}

\begin{document}

\begin{abstract}
We initiate a joining rigidity theory for rational maps on the Riemann
sphere $\mathbb P^1=\mathbb P^1(\mathbb C)$. Let
$f_1,f_2\colon\mathbb P^1\to\mathbb P^1$ be rational maps of degree
at least $2$,
and 
$\mu_1,\mu_2$
their respective measures of
maximal entropy, whose supports are the Julia sets $J(f_1)$ and
$J(f_2)$. We study ergodic joinings of the systems
$(J(f_1),f_1,\mu_1)$ and $(J(f_2),f_2,\mu_2)$, namely ergodic
probability measures on $J(f_1)\times J(f_2)$ which are invariant
under $f_1\times f_2$ and whose marginals are $\mu_1$ and $\mu_2$.

Our main theorem shows that a positive-mass local holomorphic relation
forces algebraic rigidity. More precisely, if the joining charges the
graph of a local biholomorphism, then that local relation globalizes to
an invariant algebraic curve and yields either a finite cycle of rational
graph or transpose-graph relations, or a genuinely multi-valued
invariant algebraic correspondence. If no local biholomorphic graph has
positive joining measure, then the joining generates a compact
non-discrete family of local holomorphic relations.

The proof introduces normalized inverse branch transfer maps and studies
their cluster limits. Starting from a local biholomorphic graph of
positive joining measure, recurrence and contraction of inverse branches
produce recurrent local intertwining relations. These are promoted to an algebraic relation by a local-to-global
rigidity argument in the non-Latt\`es case and by affine
uniformization in the Latt\`es case. In the absence of
any positive-mass local biholomorphic graph, the cluster family must be
infinite, and its non-discrete closure gives the second alternative.
\end{abstract}

\maketitle

\section{Introduction}
\subsection{Joinings in dynamical systems}
The study of joinings in dynamical systems
began with Furstenberg's foundational work \cite{Furstenberg1967}. A joining of two measure-preserving systems is an invariant measure on the product space whose coordinate marginals are the original invariant measures. While the product measure is always a joining, non-product joinings often reveal an underlying geometric or algebraic relation between the two systems. Furstenberg introduced this language to formalize the idea that two systems are dynamically unrelated when their only joining is the product joining, a property he called {\it disjointness}. Since then, joinings have become one of the central tools in measurable dynamics. They provide a way to detect common factors, rigidity and multiple recurrence phenomena; see, for example, the survey of de la Rue \cite{deLaRue2005} and Glasner's monograph \cite{Glasner2003}.
Furstenberg's joining theory forms part of a broader measure-rigidity
perspective, also exemplified by the $\times2,\times3$ conjecture \cite{Furstenberg1967}, in which
multiple dynamical invariances are expected to force algebraic structure.
A \emph{self-joining}
is a joining of a
system with itself and provides a natural measure-theoretic means of detecting
its internal symmetries. For example, a measure-preserving map commuting with
the dynamics determines a self-joining supported on its graph, while more
general self-joinings may be supported on 
invariant correspondences.

A central theme in joining theory is \emph{rigidity}. In many natural
systems, every ergodic joining is either the product joining or is forced by
an explicit algebraic or geometric relation. Ratner's joining theorem for
horocycle flows is a fundamental example. In her work on horocycle flows and
rigidity of products \cite{Ratner1983}, Ratner 
proved that non-trivial
joinings of finite-volume horocycle flows arise from algebraic
correspondences, providing a foundational model for joining rigidity in
homogeneous dynamics.
A higher-rank counterpart was developed by Einsiedler and Lindenstrauss.
They classified joinings for broad classes of higher-rank diagonalizable
actions on locally homogeneous spaces
\cite{EinsiedlerLindenstrauss2007} and later proved that joinings of
higher-rank torus actions on $S$-arithmetic homogeneous spaces are
algebraic \cite{EinsiedlerLindenstrauss2019}.

Joining rigidity has also been extended to infinite-volume rank-one spaces.
Mohammadi and Oh \cite{MohammadiOh2016} classified locally finite joinings of
horospherical subgroup actions on $\Gamma\backslash G$, where
$G=\mathrm{PSL}_2(\mathbb R)$ or
$\mathrm{PSL}_2(\mathbb C)$ and $\Gamma$ is geometrically finite and
Zariski dense, extending Ratner's theorem from lattices to infinite-volume
hyperbolic geometry. Pan \cite{Pan2018} obtained a joining classification for
horocycle flows on certain abelian covers, while related work of Oh--Pan
\cite{OhPan2019} established local mixing and classified horospherically
invariant measures on abelian covers.

The goal of this paper is to initiate a joining rigidity theory in
holomorphic dynamics.

\subsection{Statement of results}
For $i=1,2$, let
$f_i\colon \mathbb P^1\to\mathbb P^1$ be 
a rational map of
degree at least $2$,  $J(f_i)$ its Julia
set and  $\mu_i$ 
its measure of maximal entropy
\cite{Lyubich83,FLM83}. Thus $J(f_i)=\operatorname{supp}(\mu_i)$.
Set $F:=f_1\times f_2$, and denote by $\pi_i$,
$i=1,2$, the two coordinate projections.
A \emph{joining} of $f_1$ and $f_2$ is an $F$-invariant probability
measure $\nu$ on $J(f_1)\times J(f_2)$ such that
$(\pi_1)_*\nu=\mu_1$ and $(\pi_2)_*\nu=\mu_2$.
When $f_1=f_2=f$, we call such a measure a \emph{self-joining} of $f$.

Our main result is a structure theorem for every non-product ergodic
joining. We call the pair $(f_1,f_2)$ \emph{joining-exceptional} if there
exist an integer $n\geq1$ and an irreducible algebraic curve
$C\subset\mathbb P^1\times\mathbb P^1$ such that
$F^n(C)=C$ and both projections
$\pi_1|_C$ and $\pi_2|_C$ are dominant and have generic degree
strictly greater than one.

\begin{thm}\label{thm_main}
Let $f_1,f_2\colon \mathbb P^1\to\mathbb P^1$ be rational maps of
degree at least $2$, and let $\mu_1,\mu_2$ be their measures of
maximal entropy. Let $\nu\neq\mu_1\times\mu_2$ be an ergodic joining
of $f_1$ and $f_2$. Then exactly one of the following two cases
occurs.

\begin{enumerate}
\item[(A)]
The joining $\nu$ charges the graph 
of a local biholomorphism, i.e., there
exist non-empty open sets $U,V\subset\mathbb P^1$ and a biholomorphism
$\varphi\colon U\to V$ such that $\nu(\Gamma_\varphi)>0$,
where $\Gamma_\varphi$ is the graph of $\varphi$.
In this case exactly
one of the following algebraic alternatives holds.

\begin{enumerate}
\item[(A1)]
There exist an integer $r\geq1$ and probability measures
$\nu_0,\ldots,\nu_{r-1}$ such that
$\nu=r^{-1}\sum_{i=0}^{r-1}\nu_i$ and
$F_*\nu_i=\nu_{i+1\bmod r}$, where, for every $i$, either
$\nu_i=(\operatorname{id},A_i)_*\mu_1$ for some rational map $A_i$
satisfying $f_2^r\circ A_i=A_i\circ f_1^r$, or
$\nu_i=(B_i,\operatorname{id})_*\mu_2$ for some rational map $B_i$
satisfying $f_1^r\circ B_i=B_i\circ f_2^r$.

\item[(A2)]
The pair $(f_1,f_2)$ is joining-exceptional, and the invariant curve
$C$ in the definition can be chosen so that $\nu(C)>0$.
\end{enumerate}

\item[(B)]
The joining $\nu$ charges no graph of a local biholomorphism.
In this case $\nu$ gives rise, in the sense made precise in Theorem
\ref{thm:Cnu-infinite}, to a compact non-discrete family
$\mathcal H\subset\operatorname{Hol}(U,V)$ in the compact-open
topology, for suitable non-empty open sets $U,V\subset\mathbb P^1$,
such that every $h\in\mathcal H$ is non-constant and is locally
biholomorphic, possibly
after restriction to a smaller non-empty open subset.
\end{enumerate}
\end{thm}

As an immediate consequence, if $\deg f_1\neq\deg f_2$, the
algebraic alternatives \textup{(A1)} and \textup{(A2)} are
impossible. Hence every non-product ergodic joining charges no local
biholomorphic graph and falls into case \textup{(B)}; see Corollary
\ref{cor:different-degrees}. We also note that the trivial product joining
gives rise to the same non-discrete cluster phenomenon, although in
that case it is automatic from the transfer-map construction; see
Remark \ref{rmk:product-A2}.

\medskip

There is
an important difference
between our setting and the classical
homogeneous examples discussed above. For a rational map of degree $d$, the maximal-entropy system is
measurably isomorphic to the one-sided uniform $d$-Bernoulli shift
\cite{HeicklenHoffman,ManeBernoulli}.
Hence it admits many ergodic
self-joinings having no evident complex-analytic origin, and
an
algebraic classification of all measurable joinings cannot hold.
Theorem \ref{thm_main} instead shows that once a joining charges a
local biholomorphic graph, the resulting local relation is forced to
globalize algebraically. 
This retains the guiding principle of joining
rigidity in homogeneous dynamics while accounting for the much larger
measurable joining space present in rational dynamics; see
\cite{Ratner1983,EinsiedlerLindenstrauss2007,
EinsiedlerLindenstrauss2019,MohammadiOh2016,OhPan2019,Pan2018}.
Concrete examples of case \textup{(B)}, with either diffuse or finite
atomic conditional measures, are given in Section
\ref{sec:measurable-examples}.

\subsection{Diagonal dynamics and self-joinings}
\label{sec:self-joinings}
In the diagonal case $f_1=f_2=f$, Theorem \ref{thm_main} gives a
structure theorem for non-product ergodic self-joinings of a single
rational map. 
While case \textup{(B)} arises from the same purely
measure-theoretic phenomena as in the general setting, case
\textup{(A)} admits a more explicit description in the diagonal case.
In alternative \textup{(A1)}, the rational maps appearing in the
graph and transpose-graph components commute with an iterate of $f$.
The curves in \textup{(A2)} belong to the class of non-vertical and
non-horizontal invariant curves studied by Pakovich
\cite{PakovichInvariantCurves}.
Indeed, setting
$A:=f^n$, the curve $C$ is $(A,A)$-invariant and, since both
projections are dominant, it is neither vertical nor horizontal.
Pakovich's classification gives a more explicit description of such
curves when $A$ is not a generalized Latt\`es map.
More precisely, by
\cite[Theorem 1.2]{PakovichInvariantCurves},
there exist rational maps
$U_1,U_2,V_1,V_2\colon\mathbb P^1\to\mathbb P^1$ commuting with $A$
and an integer $d\geq0$ such that
$$
U_1\circ V_1=U_2\circ V_2=A^{\circ d},
\qquad
V_1\circ U_1=V_2\circ U_2=A^{\circ d},
$$ and the map
$t\mapsto(U_1(t),U_2(t))$ parametrizes $C$.
Thus, outside the generalized Latt\`es case in Pakovich's sense,
these diagonal joining-exceptional curves are described in terms of
rational maps commuting with an iterate of $f$.

The condition
\textup{(A2)} is designed to capture the same phenomenon as the classical
exceptional rational maps, namely the monomials, Chebyshev maps and
Latt\`es maps, which have large symmetry or semiconjugacy structures
and admit algebraic correspondences invariant under the diagonal
dynamics. In particular, the classical exceptional maps fall into
this framework; see Proposition \ref{prop:def-dim-1}.

\subsection{Joining rigidity as a unifying framework in complex dynamics}
The joining viewpoint provides a natural framework for several
rigidity phenomena that have previously been studied separately in
complex dynamics: local and global symmetries of Julia sets,
semiconjugate and
commuting rational maps, and the promotion of
measure-preserving local relations to algebraic correspondences.

Classical work of Baker--Eremenko \cite{BakerEremenko1987}, Beardon
\cite{Beardon1992}, Eremenko \cite{EremenkoCommuting}, Levin
\cite{Levin1990}, Levin--Przytycki
\cite{LevinPrzytycki1997}, and Dinh \cite{Din00}
shows that rational maps with the same Julia set, as well as local and
global symmetries of Julia sets, are strongly constrained. 

Semiconjugacies provide a natural source of relations between two
dynamical systems; see, for example,
\cite{PakovichSemiconjugate}. The relations in alternative
\textup{(A1)} are of this type.
In the
diagonal case, Pakovich's classification, discussed in Section
\ref{sec:self-joinings}, describes the invariant curves in
alternative \textup{(A2)} in terms of rational maps commuting with an
iterate. For general pairs of a fixed degree,
\cite{PakovichPeriodicCurves} shows that a general pair admits a
non-vertical, non-horizontal periodic curve if and only if the two
maps are Möbius-conjugate.

Commuting rational maps arise as the diagonal specialization of the
semiconjugacy picture. Their structure, and in particular their
connection with the classical exceptional families, goes back to Ritt
\cite{RittPermutable}
and was further developed by Eremenko and Pakovich; see
\cite{EremenkoCommuting,PakovichCommuting}.
Recent work of Beaumont \cite{Bea25} on centralizers of iterates is
closely related to alternative \textup{(A1)}.
For related results on commuting endomorphisms in higher dimension, we refer to \cite{DinhCommutingC2,DinhSibonyCommutingPk,Kau18}.

More recently, Dujardin--Favre--Gauthier \cite{DFG23} and Ji--Xie \cite{JiXieLocalRigidityJuliaSets} proved
local-to-global rigidity results in
which suitable local symmetries preserving Julia sets or measures of
maximal entropy are promoted to algebraic correspondences. 
Related
local-to-global rigidity phenomena also appear in 
\cite{JiXieHomoclinic}, where periodic and homoclinic data are used to obtain global
rigidity statements for rational maps, including a new proof of
McMullen's multiplier-rigidity theorem \cite{McMullen1987}.

Related measure and Julia set rigidity results occur throughout complex dynamics.
For
rational maps on $\mathbb P^1$, see
\cite{LevinPrzytycki1997,Ye15,PakovichMaxCritical}. Related questions
for 
regular polynomial endomorphisms of $\mathbb C^k$ are considered in
\cite{Ueno10,ZhangMeasureRigidity}, while we refer to
\cite{Lamy01,DujFav17,Vigny26} and
references therein for the case of automorphisms.
Measure-theoretic rigidity phenomena have also been studied in the setting of 
random dynamics of projective surfaces \cite{CantatDujardinRandom}.

Although the algebraic part of Theorem \ref{thm_main} belongs to this
broader circle of rigidity phenomena, its starting point is different.
We assume neither an equality of maximal-entropy measures nor a local
analytic symmetry, semiconjugacy, or algebraic correspondence.
Instead, our starting point is a {\it measure-theoretic} coupling of the
two maximal-entropy systems, with no analytic regularity assumed.
Theorem \ref{thm_main} 
shows that once such a coupling charges a local
biholomorphic graph, this local relation globalizes to an invariant
algebraic curve. Related questions and possible extensions are discussed
in Section \ref{sec:further-questions}.

\subsection{Strategy of the proof}
The proof of Theorem \ref{thm_main} consists of two main parts: a
procedure for extracting local holomorphic structure from an ergodic
joining and a local-to-global argument applied to the resulting
cluster family. When a local biholomorphic graph carries positive
joining measure, this argument promotes the local graph to an
invariant algebraic curve, leading to alternative \textup{(A1)} or
\textup{(A2)} in case \textup{(A)}; otherwise the cluster
construction itself yields the non-discrete family in case
\textup{(B)}.

For the extraction step, we construct transfer maps
by comparing
inverse branches along typical pairs of inverse orbits. 
After normalization,
the transfer maps have
derivative equal to $1$ at the base point, and Koebe distortion gives
uniform control on fixed reference domains. Hence,
after passing to fixed reference charts, the transfer maps form a
normal family. For each recurrent point of the lifted joining, we
consider the cluster limits of the corresponding normalized transfer
maps and let $\mathcal C_\nu$ be the union of these cluster sets. A
cluster germ is called {\it admissible}
if it occurs over a set of positive
lifted joining measure. Such a germ gives a local biholomorphic graph
of positive $\nu$-measure.

\smallskip
\noindent
\textbf{Positive-mass local relation.}
Suppose that $\nu$ charges the graph of a local biholomorphism
$\varphi$. Recurrence and contraction of inverse branches produce an
iterate $N\geq1$, a coordinate disc $D$, and two distinct inverse
branches $\alpha_1,\alpha_2\colon D\to D$ of $f_1^N$, with relatively
compact images, such that
$f_2^N\circ\varphi\circ\alpha_j=\varphi$ on $D$. Thus $f_1^N$ is non-injective on the union of the two return domains;
one of the contracting return branches also yields a repelling fixed
point of $f_1^N$.

When neither map is Latt\`es, the generalized form of Inou's theorem
\cite{Inou11}
proved by Ji--Xie \cite[Theorem 2.1]{JiXieLocalRigidityJuliaSets}
promotes
this local intertwining relation to an invariant irreducible algebraic
curve. If one map is Latt\`es, we show that the charged local graph first forces the
other map to be Latt\`es as well. We then use affine uniformization and the
abundance of recurrent inverse branches to show that the slope of
the lifted affine relation maps a finite-index sublattice of the first
translation lattice into the second. This again produces an invariant
algebraic curve. In either case, the projection degrees give
\textup{(A1)} or \textup{(A2)}.
Since $\mu_1$ and $\mu_2$ are non-atomic, so is $\nu$; as distinct
irreducible curves have finite intersection, alternatives
\textup{(A1)} and \textup{(A2)} are then mutually exclusive.

\smallskip
\noindent
\textbf{No positive-mass local relation.}
Suppose instead that $\nu$ charges no graph of a local
biholomorphism. Then no cluster germ can be admissible. Since a finite
cluster family necessarily contains an admissible germ,
$\mathcal C_\nu$ must be infinite. Its closure in the ambient normal
family is therefore compact and non-discrete. The uniform derivative
lower bound from the normalization shows that every element of this
closure is non-constant and locally biholomorphic after restriction to
a smaller non-empty open set. Passing back through the fixed reference
charts gives the family in case \textup{(B)}.

\smallskip

Beyond this dichotomy,
we show that recurrent returns preserve these
fiberwise cluster sets up to natural left and right coordinate changes; see Section 
\ref{sec:cluster-gauge}.

\subsection{Organization of the paper}
Section \ref{sec:transfer-cluster} develops the normalized inverse
branch construction and the associated cluster families. Sections
\ref{sec:finite-algebraization} and \ref{sec:Cnu-infinite} establish the
two cases of Theorem \ref{thm_main} and Section \ref{sec:pf-main-thm} completes its proof.
Section \ref{sec:measurable-examples} gives examples illustrating the
necessity of case \textup{(B)} for arbitrary
measure-theoretic joinings.
Section
\ref{sec:cluster-gauge} studies the finer fiberwise cluster structure,
and Section \ref{sec:classical-ex-maps} shows that the classical
exceptional rational maps are joining-exceptional. Finally, Section
\ref{sec:further-questions} discusses several directions for further
study.

\subsection*{Acknowledgments}
The authors would like to thank the 2025 Summer Collaborators Program at the Institute for Advanced Study, where work on this project began.
This project has received funding from the
project ANR TIGerS (ANR-24-CE40-3604)
and from the MIUR Excellence Department Project awarded to the
Department of Mathematics of the University of Pisa,
CUP I57G22000700001.
The first author is affiliated to the GNSAGA group of INdAM. The second author is supported by National Science Foundation DMS-2554272.

\section{Local transfer maps and cluster germs}
\label{sec:transfer-cluster}
This section carries out the first step in the proof of
Theorem \ref{thm_main}: extracting local holomorphic structure from
an ergodic joining.
Section \ref{subsec:inverse-branches} 
summarizes
inverse-branch estimates
along typical inverse orbits. In Section \ref{subsec:normalized-transfer},
we restrict to a positive-measure set on which the inverse branches in
both coordinates are defined on a common disc and construct normalized
inverse branch transfer maps. Section
\ref{subsec:fixed-reference-charts} passes to fixed reference charts
and gives a uniform derivative lower bound for cluster limits. In
Section \ref{subsec:cluster-germs}, we define the total cluster family
$\mathcal C_\nu$ and the admissible germs $\mathcal A_\nu$, and
realize every admissible germ as a local graph charged by the joining.

\smallskip

Throughout this section, we let $f_i\colon \mathbb P^1\to\mathbb P^1$,
$i=1,2$, be rational maps of degree at least $2$, 
$\mu_i$ 
their measures of maximal entropy \cite{Lyubich83,FLM83},
$J(f_i)=\operatorname{supp} \mu_i$ their Julia sets,
and $\nu$  an ergodic
joining of $f_1$ and $f_2$. 
For $i=1,2$, we let $(\hat X_i,\hat f_i)$ be the natural extension of
$(J(f_i),f_i)$ and 
$\hat\mu_i$
the lift of $\mu_i$ to
$\hat X_i$. We denote by
$\pi\colon\hat X_i\to J(f_i)$ the projection
$\pi(\hat x)=x_0$.
We will use without further comment that each
$\mu_i$ is non-atomic and mixing; 
in particular, $\mu_i$ is ergodic for $f_i^r$ for every
$r\geq1$. We will also use the standard fact that the natural extension
of an ergodic invariant measure is ergodic \cite{CFS12}.

Set
$F:=f_1\times f_2$ and
$\hat F:=\hat f_1\times\hat f_2$.
If $\nu$ is an $F$-invariant measure on
$J(f_1)\times J(f_2)$, we denote by $\hat\nu$ its lift to
$\hat X_1\times\hat X_2$. Abusing notation, we also write
$\pi\colon\hat X_1\times\hat X_2\to J(f_1)\times J(f_2)$
for the natural projection. We denote by $\pi_i$, $i=1,2$, the
coordinate projections on $\mathbb P^1\times\mathbb P^1$.

\subsection{Inverse branches and estimates along one orbit}
\label{subsec:inverse-branches}
We begin by fixing centered conformal charts with uniform local
control. Choose a finite open cover
$\{W_\alpha\}_{\alpha=1}^N$ of $\mathbb P^1$ such that, for each
$\alpha$, there is a smooth family of centered conformal charts
$$
\psi_{\alpha,p}\colon
B_{\mathrm{sph}}(p,r_0)\to\mathbb C,
\qquad
\psi_{\alpha,p}(p)=0,
\qquad
p\in W_\alpha,
$$
where $r_0>0$ is independent of $\alpha$ and $p$.

Choose open sets $W'_\alpha\Subset W_\alpha$
such that $\{W'_\alpha\}_{\alpha=1}^N$ still covers
$\mathbb P^1$.
Choose a measurable partition
$
\mathbb P^1=P_1\sqcup\cdots\sqcup P_N
$
such that
$
P_\alpha\subset W'_\alpha.
$
For $p\in P_\alpha$, set
$\psi_p:=\psi_{\alpha,p}$.
Since
$P_\alpha\subset W'_\alpha\Subset W_\alpha$ and there are only
finitely many chart families, after decreasing $r_0$ if necessary
the chosen charts $\psi_p$ and their inverses have uniform
distortion bounds on slightly smaller spherical balls, independently
of $p$.
Moreover,
on $P_\alpha$ the chosen
charts belong to the fixed smooth family
$p\mapsto\psi_{\alpha,p}$.

Fix $i\in\{1,2\}$, 
$\hat z=(z_n)_{n\in\mathbb Z}\in\hat X_i$ and 
$n\geq1$.
On a suitable neighbourhood $U_{i,\hat z,n}$ of $z_0$, denote by
$g_{i,\hat z,n}$ the inverse branch of $f_i^n$ along $\hat z$,
characterized by
$g_{i,\hat z,n}(z_0)=z_{-n}$ and
$f_i^n\circ g_{i,\hat z,n}=\operatorname{id}$.
Whenever $g_{i,\hat z,n}$ is defined on a spherical ball
$B_{\mathrm{sph}}(z_0,\rho)$, set
\begin{equation}\label{eq:G}
G_{i,\hat z,n}
:=
\psi_{z_{-n}}
\circ
g_{i,\hat z,n}
\circ
\psi_{z_0}^{-1}.
\end{equation}

The next proposition summarizes the inverse branch estimates that we will need.

\begin{prop}
\label{prop:inverse branch-package}
For every $\varepsilon>0$ and $i\in\{1,2\}$, there exist a measurable
$\hat f_i$-invariant set
$\hat X_{i,\varepsilon}\subset\hat X_i$ with
$\hat\mu_i(\hat X_{i,\varepsilon})=1$ and a measurable function
$\rho_{i,\varepsilon}\colon
\hat X_{i,\varepsilon}\to(0,+\infty)$
such that the following hold for every
$\hat z\in\hat X_{i,\varepsilon}$ and every $n\geq1$.

\begin{enumerate}
\item
The inverse branch $g_{i,\hat z,n}$ is defined on
$B_{\mathrm{sph}}(z_0,\rho_{i,\varepsilon}(\hat z))$.

\item
The map $G_{i,\hat z,n}$ defined in \eqref{eq:G} is holomorphic and
univalent on
$D(0,\rho_{i,\varepsilon}(\hat z))$.

\item
For every $0<\lambda<1$, there exists a constant
$\kappa(\lambda)\geq1$, independent of $i$, such that, for every
$u,v\in
D(0,\lambda\rho_{i,\varepsilon}(\hat z))$,
we have
$$
\kappa(\lambda)^{-1}
\leq
\left|
\frac{G'_{i,\hat z,n}(u)}
     {G'_{i,\hat z,n}(v)}
\right|
\leq
\kappa(\lambda),
$$
and
$$
\kappa(\lambda)^{-1}|G'_{i,\hat z,n}(0)|
\leq
\frac{
\operatorname{diam}
G_{i,\hat z,n}
\bigl(D(0,\lambda\rho_{i,\varepsilon}(\hat z))\bigr)
}{
\lambda\rho_{i,\varepsilon}(\hat z)
}
\leq
\kappa(\lambda)|G'_{i,\hat z,n}(0)|.
$$

\item
For every $\hat z\in\hat X_{i,\varepsilon}$, we have
$
|G'_{i,\hat z,n}(0)|\to0
$
as $n\to\infty$. Consequently, for every $0<\lambda<1$,
$
\operatorname{diam}
G_{i,\hat z,n}
(D(0,\lambda\rho_{i,\varepsilon}(\hat z)))\to0.
$
\end{enumerate}
\end{prop}

\begin{proof}
By the standard inverse branch construction for rational maps
associated to typical inverse orbits for invariant ergodic measures
with positive Lyapunov exponent, for almost every $\hat z$ there is
a positive spherical radius, depending measurably on $\hat z$, on
which all the inverse branches $g_{i,\hat z,n}$ are defined; see,
for example, \cite[Chapter 11]{PUbook} and \cite{Berteloot10}.
Since the centered charts were chosen from finitely many smooth
local families with uniform distortion bounds, after decreasing
these radii by a uniform factor we may arrange simultaneously that
$g_{i,\hat z,n}$ is defined on
$B_{\rm sph}(z_0,\rho_{i,\varepsilon}(\hat z))$ and that
$G_{i,\hat z,n}$ is defined and univalent on
$D(0,\rho_{i,\varepsilon}(\hat z))$.
This proves (1) and (2).

Fix $0<\lambda<1$. For
$\hat z\in\hat X_{i,\varepsilon}$ and $n\geq1$, consider
$$
H_{i,\hat z,n}(w)
:=
\frac{1}{\rho_{i,\varepsilon}(\hat z)}
G_{i,\hat z,n}
\bigl(\rho_{i,\varepsilon}(\hat z)w\bigr),
\qquad |w|<1.
$$
Since $G_{i,\hat z,n}$ is univalent on
$D(0,\rho_{i,\varepsilon}(\hat z))$, the map
$H_{i,\hat z,n}$ is univalent on the unit disc. The Koebe distortion
theorem on $D(0,\lambda)$ 
gives the first
estimate in (3), with a constant $\kappa(\lambda)\ge1$ depending
only on $\lambda$.
Applying 
Koebe distortion and growth estimates,
and enlarging
$\kappa(\lambda)$ if necessary, gives the diameter estimate in (3).

Finally, intersect $\hat X_{i,\varepsilon}$ with the
full-measure set of inverse orbits for which the Lyapunov exponent
is realized, and then with all its iterates under $\hat f_i$.
After this replacement, $\hat X_{i,\varepsilon}$ is
$\hat f_i$-invariant and, for every
$\hat z\in\hat X_{i,\varepsilon}$, the derivative of the
inverse branch $g_{i,\hat z,n}$ at $z_0$ tends to zero as
$n\to\infty$. Since the derivatives of the centered charts and
their inverses are uniformly bounded, it follows that
$|G'_{i,\hat z,n}(0)|\to0$. The diameter convergence in (4)
then follows from the second estimate in (3).
\end{proof}

\subsection{Uniform domains and normalized transfer maps}
\label{subsec:normalized-transfer}
The inverse-branch domains in Proposition
\ref{prop:inverse branch-package} depend on the individual inverse
orbits in the two coordinates. We now restrict to a positive-measure
set on which the inverse branches in both coordinates are
simultaneously defined on a common disc of uniform size.

\begin{lem}\label{prop:good-pairs-good-times}
Fix $\varepsilon>0$. 
There exist a measurable set
$E\subset\hat X_{1,\varepsilon}\times
\hat X_{2,\varepsilon}$ with $\hat\nu(E)>0$ and a constant
$\rho_*>0$ such that, for every
$z=(\hat x,\hat y)\in E$ and every $n\geq1$, the maps
$G_{1,\hat x,n}$ and $G_{2,\hat y,n}$ are defined and
univalent on $D(0,\rho_*)$.
Moreover, if $\nu\neq\mu_1\times\mu_2$, then $E$ can be chosen so
that
\begin{equation}\label{eq:22-extra}
\hat\nu(E)>
(\hat\mu_1\times\hat\mu_2)(E)>0.
\end{equation}
\end{lem}

\begin{proof}
For $M\geq1$, set
$$
F_M:=
\left\{
(\hat x,\hat y)\in
\hat X_{1,\varepsilon}\times\hat X_{2,\varepsilon}:
\quad
\rho_{1,\varepsilon}(\hat x)\geq\frac1M,\quad
\rho_{2,\varepsilon}(\hat y)\geq\frac1M
\right\}.
$$
The sets $F_M$ increase to
$\hat X_{1,\varepsilon}\times\hat X_{2,\varepsilon}$,
which has full $\hat\nu$-measure since the marginals of
$\hat\nu$ are $\hat\mu_1$ and $\hat\mu_2$.
Hence $\hat\nu(F_M)>0$ for all sufficiently large $M$.
Fix such an $M$
and set
$E:=F_M$ and
$\rho_*:=1/M$.
The first assertion
follows from Proposition
\ref{prop:inverse branch-package}.

Assume now that $\nu\neq\mu_1\times\mu_2$. 
Since measurable
rectangles determine a probability measure on the product,
there exist
measurable sets $A\subset J(f_1)$ and $B\subset J(f_2)$ such that
$\nu(A\times B)>\mu_1(A)\mu_2(B)$.
Disintegrating $\nu$ over the first coordinate, write
$\nu=\int\nu_x\,d\mu_1(x)$.
Then
$$
\int_A\nu_x(B)\,d\mu_1(x)>\mu_1(A)\mu_2(B).
$$
Hence there exist $\varepsilon'>0$ and a measurable set
$A'\subset A$ with $\mu_1(A')>0$ such that
$\nu_x(B)\geq\mu_2(B)+\varepsilon'$
for every $x\in A'$.
Set
$$
E_0:=
\bigl(\pi^{-1}(A')\cap\hat X_{1,\varepsilon}\bigr)
\times
\bigl(\pi^{-1}(B)\cap\hat X_{2,\varepsilon}\bigr).
$$
Since $\hat\mu_i(\hat X_{i,\varepsilon})=1$ and the
marginals of $\hat\nu$ are $\hat\mu_1$ and
$\hat\mu_2$, 
we have
$$
\begin{aligned}
\hat\nu(E_0)
& =\nu(A'\times B)
 =\int_{A'}\nu_x(B)\,d\mu_1(x) \\
& \geq
\mu_1(A')\bigl(\mu_2(B)+\varepsilon'\bigr)
 >
\mu_1(A')\mu_2(B) 
=
(\hat\mu_1\times\hat\mu_2)(E_0)>0.
\end{aligned}
$$
Since $E_0\cap F_M$ increases to $E_0$, continuity from below gives
$$\hat\nu(E_0\cap F_M)\to\hat\nu(E_0)
\qquad
\text{and}\qquad
(\hat\mu_1\times\hat\mu_2)(E_0\cap F_M)
\to
(\hat\mu_1\times\hat\mu_2)(E_0).
$$
For $M$ sufficiently large, this gives
$\hat\nu(E_0\cap F_M)>
(\hat\mu_1\times\hat\mu_2)(E_0\cap F_M)>0$.
Setting
$E:=E_0\cap F_M$ and $\rho_*:=1/M$
gives 
\eqref{eq:22-extra}.
\end{proof}

From now on, we fix a set
$E\subset
\hat X_{1,\varepsilon}\times\hat X_{2,\varepsilon}$
and a constant $\rho_*>0$ as in Lemma
\ref{prop:good-pairs-good-times}. When
$\nu\neq\mu_1\times\mu_2$, we choose $E$ so that \eqref{eq:22-extra} holds.
For
$z=(\hat x,\hat y)\in E$
and $n\geq1$, define
$$
a_n(z)
:=
\frac{G'_{2,\hat y,n}(0)}
     {G'_{1,\hat x,n}(0)}
\in\mathbb C^*
$$
and the corresponding \emph{normalized transfer map}
$$
\Phi_{z,n}
:=
G_{2,\hat y,n}^{-1}
\circ
\bigl(a_n(z)\cdot\operatorname{id}\bigr)
\circ
G_{1,\hat x,n}.
$$

\begin{prop}\label{prop:transfer-maps}
There exist radii $0<r<R<\rho_*$ such that, for every
$z=(\hat x,\hat y)\in E$ and every $n\geq1$, the map
$
\Phi_{z,n}\colon D(0,r)\to D(0,R)
$
is well-defined and holomorphic, and satisfies
$\Phi'_{z,n}(0)=1$.
\end{prop}

\begin{proof}
By Lemma \ref{prop:good-pairs-good-times}, the maps
$G_{1,\hat x,n}$ and $G_{2,\hat y,n}$ are univalent on
$D(0,\rho_*)$.
Fix $0<\lambda<1$.
By Koebe distortion and the Koebe quarter theorem, there exist
constants $c_1,c_2>0$, depending only on $\lambda$, such that
$|G_{1,\hat x,n}(u)|
\leq
c_2|G'_{1,\hat x,n}(0)|\,|u|$
for $|u|<\lambda\rho_*$, while
$G_{2,\hat y,n}(D(0,\lambda\rho_*))$ contains
$
D\bigl(
0,c_1|G'_{2,\hat y,n}(0)|\lambda\rho_*
\bigr)$.
Choose $0<r<\lambda\rho_*$ so that
$c_2r<c_1\lambda\rho_*$. 
Since
$
|a_n(z)|\,|G'_{1,\hat x,n}(0)|
=
|G'_{2,\hat y,n}(0)|$, 
for $|u|<r$, we have
$$
|a_n(z)G_{1,\hat x,n}(u)|
\leq
c_2|G'_{2,\hat y,n}(0)|\,|u|
<
c_1|G'_{2,\hat y,n}(0)|\lambda\rho_*.
$$
It follows that
$
a_n(z)G_{1,\hat x,n}(D(0,r))
\subset
G_{2,\hat y,n}(D(0,\lambda\rho_*)).
$
Thus, setting $R:=\lambda\rho_*$, $\Phi_{z,n}$ is well-defined on $D(0,r)$ and takes values in
$D(0,R)$.
Finally, we have
$
\Phi'_{z,n}(0)
=
(G_{2,\hat y,n}^{-1})'(0)\,
a_n(z)\,
G'_{1,\hat x,n}(0)
=
1$. 
\end{proof}

\begin{cor}\label{cor:normal-family}
Let $r$ and $R$ be as in Proposition \ref{prop:transfer-maps}.
For every $z\in E$, the family
$\{\Phi_{z,n}\}_{n\geq1}$
is normal on $D(0,r)$. In particular, every sequence in this family
admits a subsequence converging locally uniformly on $D(0,r)$ to a
holomorphic map
$\Psi\colon D(0,r)\to D(0,R)$.
Every such limit map is non-constant.
\end{cor}

\begin{proof}
The maps $\Phi_{z,n}$ take values in the bounded disc $D(0,R)$, so
Montel's theorem gives normality. If
$\Phi_{z,n_j}\to\Psi$ locally uniformly, then
$
\Phi'_{z,n_j}(0)\to\Psi'(0)$.
By Proposition \ref{prop:transfer-maps},
$\Phi'_{z,n_j}(0)=1$, and hence $\Psi'(0)=1$.
Thus $\Psi$ is non-constant. Since $\Psi$ is a locally uniform limit
of maps with values in $D(0,R)$, the maximum principle gives
$\Psi(D(0,r))\subset D(0,R)$.
\end{proof}

\subsection{Fixed reference charts}
\label{subsec:fixed-reference-charts}
The transfer maps of Proposition \ref{prop:transfer-maps} are written
in moving centered charts attached to $(x_0,y_0)$. We now conjugate
them to a fixed pair of reference charts on a positive-measure subset
of $E$.

\begin{prop}\label{prop:reference-charts}
There exist points
$a\in J(f_1)$ and $b\in J(f_2)$,
centered
conformal reference charts $\chi_a$ and $\chi_b$,
spherical neighbourhoods
$U\ni a$ and $V\ni b$,
a measurable set $E^\sharp\subset E$ with
$\hat\nu(E^\sharp)>0$,
radii $0<r^\sharp<r$ and $R^\sharp>0$, compact families
$\mathcal H
\subset
\operatorname{Hol}(D(0,r^\sharp),\mathbb C)$
and
$\mathcal G
\subset
\operatorname{Hol}(D(0,R),\mathbb C)$,
and a constant $c_*>0$ such that the following hold.

\begin{enumerate}
\item
We have
$
\overline{\chi_a(U)}
\Subset D(0,r^\sharp)$ and
$\overline{D(0,R^\sharp)}
\Subset
\chi_b(V)$.
For every
$z=(\hat x,\hat y)\in E^\sharp$, we have
$x_0\in U$ and $y_0\in V$. 
The map
$h_{\hat x}
:=
\psi_{x_0}\circ
\chi_a^{-1}$ 
is well-defined on $D(0,r^\sharp)$ and belongs to
$\mathcal H$, while the
map
$k_{\hat y}
:=
\chi_b\circ
\psi_{y_0}^{-1}$
is well-defined on a fixed
neighbourhood of $\overline{D(0,R)}$ and its restriction to
$D(0,R)$ belongs to $\mathcal G$.

\item
For every $z\in E^\sharp$ and every $n\geq1$, the conjugated transfer
map
$$
\widetilde\Phi_{z,n}
:=
k_{\hat y}\circ\Phi_{z,n}\circ h_{\hat x}
$$
is well-defined and holomorphic on $D(0,r^\sharp)$, with values in
$D(0,R^\sharp)$.

\item
For every $z\in E^\sharp$, the family
$\{\widetilde\Phi_{z,n}\}_{n\geq1}$ is normal on
$D(0,r^\sharp)$
and every cluster limit
$
\widetilde\Psi\colon
D(0,r^\sharp)\to D(0,R^\sharp)
$
is non-constant.

\item
For every
$z=(\hat x,\hat y)\in E^\sharp$, setting
$q_z:= \chi_a(x_0)$, we have
$
|
k_{\hat y}'(0)\,
h_{\hat x}'(q_z)
|
\geq c_*.
$
Consequently, if $\widetilde\Psi$ is any cluster limit of
$\{\widetilde\Phi_{z,n}\}_{n\geq1}$, then
$
|\widetilde\Psi'(q_z)|\geq c_*.
$
\end{enumerate}
Moreover, if
$\hat\nu(E)>
(\hat\mu_1\times\hat\mu_2)(E)$,
then $E^\sharp$ can be chosen so that
\begin{equation}\label{eq:25-extra}
\hat\nu(E^\sharp)>
(\hat\mu_1\times\hat\mu_2)(E^\sharp).
\end{equation}
\end{prop}

\begin{proof}
Recall that the centered charts in Section
\ref{subsec:inverse-branches} were chosen from finitely many smooth local
families. More precisely, there are open sets
$W_\alpha\subset\mathbb P^1$, smaller open sets
$W'_\alpha$ with
$\overline{W'_\alpha}\Subset W_\alpha$, and a measurable partition
$
\mathbb P^1=P_1\sqcup\cdots\sqcup P_N
$
such that $P_\alpha\subset W'_\alpha$ and
$\psi_p=\psi_{\alpha,p}$ whenever $p\in P_\alpha$.

For $\alpha,\beta\in\{1,\ldots,N\}$, set
$Q_{\alpha,\beta}
:=
(P_\alpha\times P_\beta)
\cap
\bigl(J(f_1)\times J(f_2)\bigr)$.
These sets form a finite measurable partition of
$J(f_1)\times J(f_2)$. Hence
$$
E
=
\bigsqcup_{\alpha,\beta}
\bigl(E\cap\pi^{-1}(Q_{\alpha,\beta})\bigr).
$$
Since $\hat\nu(E)>0$, there is a pair $(\alpha,\beta)$ such that, setting
$
E_{\alpha,\beta}:=
E\cap\pi^{-1}(Q_{\alpha,\beta})$,
we have $\hat\nu(E_{\alpha,\beta})>0$.
If in addition
$\hat\nu(E)>
(\hat\mu_1\times\hat\mu_2)(E)$,
we choose $(\alpha,\beta)$ so that
$
\hat\nu(E_{\alpha,\beta})>
(\hat\mu_1\times\hat\mu_2)(E_{\alpha,\beta});
$
such a pair exists by the finite partition above.
On the first coordinate of $E_{\alpha,\beta}$, all chosen centered
charts belong to the single smooth family
$p\mapsto\psi_{\alpha,p}$, while on the second coordinate they belong
to the single smooth family
$q\mapsto\psi_{\beta,q}$.
Set
$$
K_{\alpha,\beta}
:=
\bigl(\overline{P_\alpha}\cap J(f_1)\bigr)
\times
\bigl(\overline{P_\beta}\cap J(f_2)\bigr).
$$
Since
$\overline{P_\alpha}\Subset W_\alpha$ and
$\overline{P_\beta}\Subset W_\beta$, the set
$K_{\alpha,\beta}$ is compactly contained in
$W_\alpha\times W_\beta$.

By compactness of $K_{\alpha,\beta}$ and the smooth dependence
of the two fixed chart families, we may choose finitely many points
$(a_j,b_j)\in K_{\alpha,\beta}$ and spherical neighbourhoods
$U_j\ni a_j$, $V_j'\ni b_j$ 
with $U_j\Subset W_\alpha$ and 
$V_j'\Subset V_j\Subset W_\beta$,
such that the sets $U_j\times V_j'$ cover $K_{\alpha,\beta}$ and,
setting
$\chi_{a_j}:=\psi_{\alpha,a_j}$ and 
$\chi_{b_j}:=\psi_{\beta,b_j}$,
the following properties hold.

\smallskip

There exist radii $0<r_j^\sharp<r$ and $R_j^\sharp>0$ such that,
for every $p\in U_j$ and $q\in V_j'$, the transition map
$\psi_{\alpha,p}\circ\chi_{a_j}^{-1}$ is defined on
$D(0,r_j^\sharp)$, while
$\chi_{b_j}\circ\psi_{\beta,q}^{-1}$ is defined on a fixed
neighbourhood of $\overline{D(0,R)}$. Moreover, we have
$$
\chi_{a_j}(U_j)\Subset D(0,r_j^\sharp),
\qquad
D(0,R_j^\sharp)\Subset\chi_{b_j}(V_j),
$$
and
$$
(\psi_{\alpha,p}\circ\chi_{a_j}^{-1})(D(0,r_j^\sharp))
\subset D(0,r),
\qquad
(\chi_{b_j}\circ\psi_{\beta,q}^{-1})(\overline{D(0,R)})
\subset D(0,R_j^\sharp)
$$
for every $p\in U_j$ and $q\in V_j'$. The corresponding families
of transition maps, restricted to the indicated fixed domains, are
relatively compact in the compact-open topology.

\smallskip

Choose a measurable partition
$Q_1,\ldots,Q_s$ of $Q_{\alpha,\beta}$ 
so that 
$Q_j\subset U_j\times V_j'$. 
Since
$$
E_{\alpha,\beta}
=
\bigsqcup_{j=1}^s
\left(E_{\alpha,\beta}\cap\pi^{-1}(Q_j)\right)
$$
and $\hat\nu(E_{\alpha,\beta})>0$, there exists $j_0$ such that
$
\hat\nu
\left(E_{\alpha,\beta}\cap\pi^{-1}(Q_{j_0})\right)>0$.
If
$\hat\nu(E_{\alpha,\beta})>
(\hat\mu_1\times\hat\mu_2)(E_{\alpha,\beta})$,
we choose $j_0$ so that
$$
\hat\nu
\left(E_{\alpha,\beta}\cap\pi^{-1}(Q_{j_0})\right)
>
(\hat\mu_1\times\hat\mu_2)
\left(E_{\alpha,\beta}\cap\pi^{-1}(Q_{j_0})\right).
$$

Set
$
a:=a_{j_0}$,
$b:=b_{j_0}$,
$U:=U_{j_0}$,
$V:=V_{j_0}$,
$r^\sharp:=r_{j_0}^\sharp$,
$R^\sharp:=R_{j_0}^\sharp$,
and
$E^\sharp
:=
E_{\alpha,\beta}\cap\pi^{-1}(Q_{j_0})$.
Then $\hat\nu(E^\sharp)>0$. If
$\hat\nu(E)>
(\hat\mu_1\times\hat\mu_2)(E)$, the
choices above also give \eqref{eq:25-extra}.
By construction,
for every $z=(\hat x,\hat y)\in E^\sharp$,
the map $h_{\hat x}$ is well-defined on
$D(0,r^\sharp)$ and, varying $z\in  E^\sharp$,
the maps $h_{\hat x}$
form a relatively compact family there.
Likewise, the maps $k_{\hat y}$ are well-defined on a fixed neighbourhood
of $\overline{D(0,R)}$, and their restrictions to $D(0,R)$ form a
relatively compact family. Let $\mathcal H$ and $\mathcal G$ be the
respective compact closures. This proves (1).

By construction, we have
$
h_{\hat x}(D(0,r^\sharp))\subset D(0,r)
$
for every $z\in E^\sharp$.
By the choice of $R^\sharp$, we have
$k_{\hat y}(D(0,R))
\subset D(0,R^\sharp)$.
Since Proposition \ref{prop:transfer-maps} gives
$\Phi_{z,n}(D(0,r))\subset D(0,R)$, the map $\widetilde\Phi_{z,n}$
is well-defined on $D(0,r^\sharp)$ and takes values in
$D(0,R^\sharp)$. This proves (2).

The family
$\{\widetilde\Phi_{z,n}\}_{n\geq1}$ takes values in the fixed bounded
disc $D(0,R^\sharp)$, so Montel's theorem gives normality. Set
$q_z:=\chi_a
(x_0)$. Since
$h_{\hat x}(q_z)=0$ and
$\Phi_{z,n}'(0)=1$, we have
$\widetilde\Phi_{z,n}'(q_z)
=
k_{\hat y}'(0)\,
h_{\hat x}'(q_z)$.
The right hand side is independent of $n$ and is non-zero. Hence, if
$\widetilde\Phi_{z,n_j}\to\widetilde\Psi$ locally uniformly, then
\begin{equation}\label{eq:derivative-tilde-psi}
\widetilde\Psi'(q_z)
=
k_{\hat y}'(0)\,
h_{\hat x}'(q_z)
\neq0.
\end{equation}
Thus every cluster limit is non-constant.
Since $\widetilde\Psi$ is a locally uniform limit of maps with
values in $D(0,R^\sharp)$, the maximum principle also gives
$\widetilde\Psi(D(0,r^\sharp))\subset D(0,R^\sharp)$.
This proves (3).

To prove (4), it remains to obtain a lower bound independent of $z$. Since
$\overline U\Subset W_\alpha$ and
$\overline V\Subset W_\beta$, the function
$$(x,y)\mapsto
\left|
(\chi_b\circ\psi_{\beta,y}^{-1})'(0)
(\psi_{\alpha,x}\circ\chi_a^{-1})'(\chi_a(x))
\right|$$
is well-defined, continuous, and strictly positive on the compact set
$\overline U\times\overline V$. It therefore has a positive minimum,
say $c_*>0$.
For $z=(\hat x,\hat y)\in E^\sharp$, the above displayed function evaluated
at $(x_0,y_0)$ is exactly
$|
k_{\hat y}'(0)\,
h_{\hat x}'(q_z)|.
$
Hence, we have
$|
k_{\hat y}'(0)\,
h_{\hat x}'(q_z)|
\geq c_*$.
Combining this with
\eqref{eq:derivative-tilde-psi}
gives
$|\widetilde\Psi'(q_z)|\geq c_*$ for every cluster limit
$\widetilde\Psi$ associated to $z$. This proves (4) and completes the proof.
\end{proof}

\subsection{Cluster germs for the joining}
\label{subsec:cluster-germs}
We fix from now on the reference data
$E^\sharp,a,b, \chi_a, \chi_b, U,V,r^\sharp$ and $R^\sharp$ obtained in
Proposition \ref{prop:reference-charts}. When $\nu\neq\mu_1\times\mu_2$, we make this choice
so that
$\hat\nu(E^\sharp)>
(\hat\mu_1\times\hat\mu_2)(E^\sharp)$.
The cluster families constructed below depend on these auxiliary
reference data, but for simplicity we will suppress this dependence from the notation.

By Proposition \ref{prop:reference-charts}, the maps
$\widetilde\Phi_{z,n}$ are defined on the fixed disc
$D(0,r^\sharp)$ and take values in $D(0,R^\sharp)$.
Set
$$
\mathcal K
:=
\overline{
\{
\widetilde\Phi_{z,n}:
z\in E^\sharp,\ n\geq1
\}
}
\subset
\operatorname{Hol}(D(0,r^\sharp),\mathbb C),
$$
where the closure is taken in the compact-open topology.
By Montel's theorem, $\mathcal K$ is compact, 
and it is metrizable
in the compact-open topology.
For each $z\in E^\sharp$, define the \emph{cluster set}
$$
C(z)
:=
\left\{
\widetilde\Psi\in\mathcal K:
\widetilde\Phi_{z,n_j}
\to
\widetilde\Psi
\text{ locally uniformly for some }n_j\to\infty
\right\}.
$$
The set $C(z)$ is non-empty and compact. Every
$\widetilde\Psi\in C(z)$ is non-constant and maps
$D(0,r^\sharp)$ into $D(0,R^\sharp)$.

We now restrict to recurrent points. Set $S:=\hat F^{-1}$.
Let $E^\sharp_{\rm rec}\subset E^\sharp$ be the recurrent part for
$S$ on $E^\sharp$, let
$\tau \colon E^\sharp_{\rm rec}\to\mathbb N$ be the first return time to
$E^\sharp$, and set
$T(z):=S^{\tau(z)}z=\hat F^{-\tau(z)}z$.
By Poincaré recurrence, we have
$\hat\nu(E^\sharp_{\rm rec})
=
\hat\nu(E^\sharp)>0$.

\begin{defn}\label{def:Cnu}
The \emph{cluster family associated to the joining} is
$$
\mathcal C_\nu
:=
\bigcup_{z\in E^\sharp_{\mathrm{rec}}}
C(z)
\subset\mathcal K.
$$
\end{defn}

The set $\mathcal C_\nu$ need not be closed. Its closure
$\overline{\mathcal C_\nu}\subset\mathcal K$ will play an important
role when $\mathcal C_\nu$ is infinite.

\smallskip

The following definition singles out cluster germs occurring on a
set of positive lifted joining measure. Section
\ref{sec:cluster-gauge} will compare 
the set of admissible germs
with the fiberwise cluster structure.

\begin{defn}\label{def:admissible-germs}
A holomorphic map
$
\widetilde\Psi
\in
\operatorname{Hol}
(D(0,r^\sharp),D(0,R^\sharp))
$
is called \emph{$\nu$-admissible} if
$$
\hat\nu
\bigl(
\{
z\in E^\sharp_{\mathrm{rec}}:
\widetilde\Psi\in C(z)
\}
\bigr)>0.
$$
\end{defn}

The set appearing in Definition
\ref{def:admissible-germs} is measurable. Indeed, for every fixed
$\widetilde\Psi\in\mathcal K$, it is equal to
$$
E^\sharp_{\rm rec}\cap
\bigcap_{m,N\geq1}
\bigcup_{n\geq N}
\left\{
z\in E^\sharp:
d_{\mathcal K}
(\widetilde\Phi_{z,n},\widetilde\Psi)
<
\frac1m
\right\}.
$$
For every fixed $n$, the map
$z\mapsto\widetilde\Phi_{z,n}$ is measurable in the compact-open
topology: its evaluations on a fixed countable dense subset of
$D(0,r^\sharp)$ depend measurably on the finite backward coordinates
$(x_0,\ldots,x_{-n},y_0,\ldots,y_{-n})$.
We denote by $\mathcal A_\nu$ the set of all
$\nu$-admissible germs.
Thus $\mathcal A_\nu\subset\mathcal C_\nu$. 

\smallskip

The following proposition realizes an admissible cluster germ as a
local holomorphic graph carrying positive joining measure.

\begin{prop}\label{prop:realization-admissible-germs}
Take $\widetilde\Psi\in\mathcal A_\nu$ and set
$
E^\sharp_{\widetilde\Psi}
:=
\{
z\in E^\sharp_{\mathrm{rec}}:
\widetilde\Psi\in C(z)
\}.
$
Then
$\varphi_{\widetilde\Psi}
:=\chi_b^{-1}\circ\widetilde\Psi\circ\chi_a$
is holomorphic on $U$. Moreover, for every
$z=(\hat x,\hat y)\in E^\sharp_{\widetilde\Psi}$,
we have
$\varphi_{\widetilde\Psi}(x_0)=y_0$ and $\varphi'_{\widetilde\Psi}(x_0)\neq0.$

Consequently, we have
$\pi(E^\sharp_{\widetilde\Psi})
\subset\Gamma_{\varphi_{\widetilde\Psi}}$
and
$\nu(\Gamma_{\varphi_{\widetilde\Psi}})>0$.
In particular, at every
$(x_0,y_0)\in\pi(E^\sharp_{\widetilde\Psi})$, the map
$\varphi_{\widetilde\Psi}$ is biholomorphic after restriction to
sufficiently small neighbourhoods of $x_0$ and $y_0$.
\end{prop}

\begin{proof}
Fix
$z=(\hat x,\hat y)\in E^\sharp_{\widetilde\Psi}$.
There exists $n_j\to\infty$ such that
$
\widetilde\Phi_{z,n_j}
\to
\widetilde\Psi
$
locally uniformly.
Since
$h_{\hat x}(
\chi_a
(x_0))=0$
and
$\Phi_{z,n_j}(0)=0$, we have
$
\widetilde\Phi_{z,n_j}(
\chi_a(x_0))
=
k_{\hat y}(0)
=
\chi_b(y_0)$.
Passing to the limit gives
$
\varphi_{\widetilde\Psi}(x_0)=y_0$.
Moreover, Proposition \ref{prop:reference-charts} (4) gives
$$
\left|
\widetilde\Psi'(
\chi_a
(x_0))
\right|
=
\left|
k_{\hat y}'(0)\,
h_{\hat x}'(
\chi_a(x_0))
\right|
\geq c_*>0.
$$
Hence
$\varphi'_{\widetilde\Psi}(x_0)\neq0$.

Thus
$\pi(E^\sharp_{\widetilde\Psi})
\subset\Gamma_{\varphi_{\widetilde\Psi}}$.
Since
$E^\sharp_{\widetilde\Psi}
\subset
\pi^{-1}(\Gamma_{\varphi_{\widetilde\Psi}})$
and $\pi_*\hat\nu=\nu$, we have
$$
\nu(\Gamma_{\varphi_{\widetilde\Psi}})
=
\hat\nu\bigl(
\pi^{-1}(\Gamma_{\varphi_{\widetilde\Psi}})
\bigr)
\geq
\hat\nu(E^\sharp_{\widetilde\Psi})
>0.
$$
The final assertion follows from the inverse function theorem.
\end{proof}

\begin{lem}\label{lem:Cnu-finite-admissible}
If $\mathcal C_\nu$ is finite, then
$\mathcal A_\nu$ is finite and non-empty.
\end{lem}

\begin{proof}
Write
$
\mathcal C_\nu
=
\{
\widetilde\Psi_1,\ldots,\widetilde\Psi_N
\}.
$
Since $C(z)$ is non-empty for every
$z\in E^\sharp_{\mathrm{rec}}$, we have
$$
E^\sharp_{\mathrm{rec}}
=
\bigcup_{j=1}^N
\left\{
z\in E^\sharp_{\mathrm{rec}}:
\widetilde\Psi_j\in C(z)
\right\}.
$$
Since $\hat \nu (E^\sharp_{\mathrm{rec}})>0$,
for
some $j$
we have
$\hat\nu
( \{
z\in E^\sharp_{\mathrm{rec}}:
\widetilde\Psi_j\in C(z)
\})
>0$.
Thus
$\widetilde\Psi_j\in\mathcal A_\nu$.
Since
$\mathcal A_\nu\subset\mathcal C_\nu$, the set
$\mathcal A_\nu$ is finite and non-empty.
\end{proof}

\begin{rmk}\label{rmk-Cnu-finite-product}
If $\nu=\mu_1\times\mu_2$, then the cluster family
$\mathcal C_\nu$ is necessarily infinite, for every fixed choice of
reference data as above. Indeed, if $\mathcal C_\nu$ were finite,
Lemma
\ref{lem:Cnu-finite-admissible}
would give a $\nu$-admissible cluster germ
$\widetilde\Psi$, and Proposition
\ref{prop:realization-admissible-germs} would give
$\nu(\Gamma_{\varphi_{\widetilde\Psi}})>0$.
On the other hand, since $\mu_2$ is non-atomic, Fubini's theorem
gives 
$$
(\mu_1\times\mu_2)(\Gamma_{\varphi_{\widetilde\Psi}})
=
\int_U
\mu_2\bigl(\{\varphi_{\widetilde\Psi}(x)\}\bigr)\,d\mu_1(x)
=0.
$$
This is a contradiction.
\end{rmk}

\section{Rigidity of positive-mass local holomorphic relations}
\label{sec:finite-algebraization}
The main result of this section is the following local-to-global
rigidity theorem.

\begin{thm}\label{thm:Cnu-finite}
Let $\nu$ be an ergodic joining of $f_1$ and $f_2$. Suppose that
there exist non-empty open sets $U,V\subset\mathbb P^1$ and a
biholomorphism $\varphi\colon U\to V$ such that
$\nu(\Gamma_\varphi)>0$. Then there exist $N\geq1$ and an
irreducible algebraic curve
$C\subset\mathbb P^1\times\mathbb P^1$ such that $F^N(C)=C$ and
$\nu(C)>0$. Moreover, exactly one of alternatives \textup{(A1)} and \textup{(A2)}
holds.
\end{thm}

The proof is organized as follows. In Section
\ref{subsec:finite-recurrent-returns}, we use recurrence and
contraction of inverse branches to obtain many contracting local
return branches from the charged graph. When neither $f_1$ nor
$f_2$ is Latt\`es (i.e.,
a finite quotient of an affine
endomorphism of an elliptic curve), the generalized Inou theorem
\cite[Theorem 2.1]{JiXieLocalRigidityJuliaSets} globalizes the
resulting local relation. In Section \ref{subsec:finite-lattes}, we
treat the remaining case: if one map is Latt\`es, then so is the
other, and affine uniformization together with a lattice-counting
argument algebraizes the charged local graph. Finally, Section
\ref{subsec:finite-classification} classifies the resulting
invariant curve and completes the proof.

For the rest of this section, we fix non-empty open sets
$U,V\subset\mathbb P^1$ and a biholomorphism
$\varphi\colon U\to V$ such that $\nu(\Gamma_\varphi)>0$. After restricting around a point in the support of
$\nu|_{\Gamma_\varphi}$, we may assume, keeping the same notation,
that $U$ is a connected coordinate disc, $V=\varphi(U)$, and
$\varphi\colon U\to V$ is biholomorphic with
$\nu(\Gamma_\varphi)>0$.

\subsection{Recurrent local relations and the non-Latt\`es case}
\label{subsec:finite-recurrent-returns}
The main result of this subsection is the following proposition.

\begin{prop}\label{prop:finite-case-algebraization}

Assume that neither $f_1$ nor $f_2$ is a Latt\`es map. Then there
exist $N\geq1$ and an irreducible algebraic curve
$C\subset\mathbb P^1\times\mathbb P^1$ such that
$(f_1^N\times f_2^N)(C)=C$ and $\nu(C)>0$.
\end{prop}

The following lemmas use the standing charged-graph hypothesis, but
do not require either map to be non-Lattès. The non-Lattès assumption
is used only in the proof of Proposition
\ref{prop:finite-case-algebraization} at the end of the subsection.

\begin{lem}\label{lem:uniform-charged-inverse-box}
There exist coordinate discs
$
D^-\Subset D\Subset U
$
and a compact set
$$
K\subset
\{z=(\hat x,\hat y):x_0\in D^-,
\ y_0=\varphi(x_0)\}
$$
of positive $\hat\nu$-measure 
such that for every sufficiently large $n$ and every
$z=(\hat x,\hat y)\in K$, the inverse branch of $f_1^n$
determined by the backward history $\hat x$ is defined on $D$.
Moreover, the diameters of the corresponding inverse-branch images
of $D$ tend to zero uniformly on $K$ as $n\to\infty$.
\end{lem}

\begin{proof}
Set $\hat\Gamma_\varphi:=\pi^{-1}(\Gamma_\varphi)$.
Intersect $\hat\Gamma_\varphi$ with the full-measure
inverse-branch set from Proposition
\ref{prop:inverse branch-package}. Since
$\rho_{1,\varepsilon}(\hat x)>0$ almost everywhere, there exists
$\rho>0$ such that
$S:=
\{(\hat x,\hat y)\in\hat\Gamma_\varphi:
\rho_{1,\varepsilon}(\hat x)\geq\rho\}$
has positive $\hat\nu$-measure.
Choose $x$ in the support of the pushforward of
$\hat\nu|_S$ under $(\hat x,\hat y)\mapsto x_0$.
Choose coordinate discs $D^-\Subset D\Subset U$ containing $x$
and a constant $0<\lambda<1$ so that
$D\subset B_{\rm sph}(x_0,\lambda\rho)$ for every $x_0\in D^-$.
Since $x$ belongs to the support of the pushforward of
$\hat\nu|_S$, the set
$S^-:=\{z\in S:x_0\in D^-\}$ has positive
$\hat\nu$-measure.

For every $z=(\hat x,\hat y)\in S^-$ and every $n\geq1$, the
inverse branch of $f_1^n$ determined by $\hat x$ is defined on $D$.
By Proposition \ref{prop:inverse branch-package} (3)--(4), together
with the uniform distortion bounds for the chosen charts, the
diameters of the corresponding images of $D$ tend to zero for every
$z\in S^-$. Egorov's theorem gives a positive-measure subset on
which this convergence is uniform, and regularity of $\hat\nu$
allows us to choose a compact subset $K$ of positive measure.
By construction,
$K$ satisfies
the required properties.
\end{proof}

\begin{lem}\label{lem:many-contracting-returns}
There exist a constant $\delta>0$ and arbitrarily large integers
$N$ for which there are a measurable set $A_N\subset D^-$ with
$\mu_1(A_N)>0$ and a collection $\mathscr B_N$ of at least
$\delta d_1^N$ distinct inverse branches
$\alpha\colon D\to D$ of $f_1^N$, where $d_1:=\deg f_1$, such that $\alpha(D)\Subset D$ and
$f_2^N\circ\varphi\circ\alpha=\varphi$ on $D$ for every
$\alpha\in\mathscr B_N$.
\end{lem}

\begin{proof}
Take $D^-\Subset D$ and $K$ from Lemma
\ref{lem:uniform-charged-inverse-box}. Since $\hat F$ is ergodic,
the mean ergodic theorem gives $c>0$ and arbitrarily large $N$ such
that $\hat\nu(K\cap\hat F^{-N}K)\geq c$.
Fix one such $N$, sufficiently large that the uniform contraction in
Lemma \ref{lem:uniform-charged-inverse-box} applies,
and set $B_N:=K\cap\hat F^{-N}K$ and
$E_N:=\{x_0:z=(\hat x,\hat y)\in B_N\}$.
Since the distribution of $x_0$ under $\hat\nu$ is $\mu_1$,
we have $\mu_1(E_N)\geq c$.
For every $u\in E_N$,
choosing $z=(\hat x,\hat y)\in B_N$ with $x_0=u$
gives
$u,f_1^N(u)\in D^-$ and
$f_2^N(\varphi(u))=\varphi(f_1^N(u))$.

For $\mu_1$-almost every $x$, set
$m_N(x):=\#(E_N\cap(f_1^N)^{-1}(x))$, with preimages counted with
multiplicity. Since $f_1^*\mu_1=d_1\mu_1$,
we have
$\int m_N\,d\mu_1=d_1^N\mu_1(E_N)\ge cd_1^N$. As
$0\le m_N\le d_1^N$, the set
$X_N:=\{x:m_N(x)\ge(c/2)d_1^N\}$ has positive
$\mu_1$-measure. Moreover, $X_N\subset D^-$ up to a
$\mu_1$-null set, as $f_1^N(E_N)\subset D^-$.

Remove the critical values of $f_1^N$ from $X_N$. Choose a simply
connected coordinate disc $Q\Subset D^-$, disjoint from these
critical values, such that $\mu_1(X_N\cap Q)>0$. The inverse branches
of $f_1^N$ on $Q$ form a finite collection. Partition $X_N\cap Q$
according to the subcollection of branches $\alpha$ for which
$\alpha(x)\in E_N$. On one element $A_N$ of this finite measurable
partition, we have $\mu_1(A_N)>0$ and a fixed collection
$\mathscr B_N$ of at least $(c/2)d_1^N$ distinct inverse branches
such that $\alpha(A_N)\subset E_N$ for every
$\alpha\in\mathscr B_N$.

For $x\in A_N$ and $\alpha\in\mathscr B_N$, 
choose $z=(\hat x,\hat y)\in B_N$ with $x_0=\alpha(x)$.
Then $\hat F^N(z)\in K$, and its first
backward history determines an inverse branch which agrees with
$\alpha$ near $x$. By the uniform inverse-branch property in
Lemma \ref{lem:uniform-charged-inverse-box},
this branch is defined on the fixed
disc $D$. By uniqueness of analytic continuation, it gives an
extension of $\alpha$ from $Q$ to $D$, which we continue to denote
by $\alpha$. 
By the choice of $N$,
the uniform contraction
in Lemma \ref{lem:uniform-charged-inverse-box} makes the diameter of
each $\alpha(D)$ smaller than
$\rho:=\operatorname{dist}(\overline{D^-},\partial D)>0$. Since
$\alpha(A_N)\subset E_N\subset D^-$, we obtain
$\alpha(D)\Subset D$ for every $\alpha\in\mathscr B_N$. In particular,
$f_2^N\circ\varphi\circ\alpha$ is holomorphic on $D$.
For $x\in A_N$, we have $\alpha(x)\in E_N$, and hence
$f_2^N(\varphi(\alpha(x)))=\varphi(x)$. The holomorphic maps
$f_2^N\circ\varphi\circ\alpha$ and $\varphi$ therefore agree on the
positive-$\mu_1$-measure set $A_N$. If they were not identical,
their coincidence set would be discrete and hence $\mu_1$-null,
since $\mu_1$ is non-atomic. Thus
$f_2^N\circ\varphi\circ\alpha=\varphi$ on $D$.
Taking $\delta:=c/2$ proves the lemma.
\end{proof}

\begin{lem}\label{lem:charged-graph-return-configuration}
There exist an integer $N\geq1$, a coordinate disc $D\Subset U$,
and an open set $W\Subset D$ such that
$f_1^N\colon W\to D$ is non-injective,
$\varphi\circ f_1^N=f_2^N\circ\varphi$
on $W$,
and $f_1^N|_W$ has a repelling fixed point.
\end{lem}

\begin{proof}
Choose a sufficiently large return time $N$ in Lemma
\ref{lem:many-contracting-returns} so that
$\delta d_1^N\geq2$, and choose two distinct branches
$\alpha_1,\alpha_2\in\mathscr B_N$. Set
$
W:=\alpha_1(D)\cup\alpha_2(D).
$
Then $W\Subset D$, the map $f_1^N\colon W\to D$ is non-injective,
and
$\varphi\circ f_1^N=f_2^N\circ\varphi$ on $W$.

Since $\alpha_1(D)\Subset D$, the strict self-map
$\alpha_1\colon D\to D$ has a fixed point $p\in D$.
By Schwarz--Pick,
$|\alpha_1'(p)|<1$. Since $\alpha_1$ is an inverse branch of
$f_1^N$, we have
$f_1^N(p)=p$ and
$|(f_1^N)'(p)|=|\alpha_1'(p)|^{-1}>1$.
Thus $p\in W$ is a repelling fixed point of $f_1^N$.
\end{proof}

We can now prove Proposition \ref{prop:finite-case-algebraization}.

\begin{proof}[Proof of Proposition \ref{prop:finite-case-algebraization}]
Take $D,W,\varphi$ and $N$ from Lemma \ref{lem:charged-graph-return-configuration}. The open set $D$ is connected,
$W\subset D$, the map $f_1^N\colon W\to D$ is non-injective and has
a repelling fixed point by Lemma \ref{lem:charged-graph-return-configuration}, and
$\varphi$ is non-constant with
$\varphi\circ f_1^N=f_2^N\circ\varphi$ on $W$.

Apply \cite[Theorem 2.1]{JiXieLocalRigidityJuliaSets} with
$f=f_1^N$, $g=f_2^N$, $U=D$, $U'=W$, and
$h=\varphi|_D$.
Since $f_1^N$ and $f_2^N$ are non-Latt\`es, the
theorem gives an irreducible algebraic curve
$C\subset\mathbb P^1\times\mathbb P^1$ invariant under
$f_1^N\times f_2^N$ and containing the graph of $\varphi$ on $D$.
Moreover, the compact set $K$ in Lemma
\ref{lem:uniform-charged-inverse-box} has positive $\hat\nu$-measure
and projects into the graph of $\varphi$ over $D^-\Subset D$.
Therefore
$\nu(\Gamma_{\varphi|_D})\geq\hat\nu(K)>0$, and hence
$\nu(C)>0$.
\end{proof}

\subsection{The Latt\`es case}
\label{subsec:finite-lattes}

We now assume that one of $f_1,f_2$ is Latt\`es. The charged local
graph first forces the other map to be Latt\`es as well; see Lemma
\ref{lem:lattes-propagation}. We then use affine uniformizations and
the many-return statement in Lemma \ref{lem:many-contracting-returns}
to algebraize the local relation.
The affine-lattice part of the argument is reminiscent of rigidity
for higher-rank algebraic actions, although those results do not
apply to the present rank-one, non-invertible setting; see, for
example,
\cite{KalininKatok,KalininSpatzier}.
The main result of this subsection is the following proposition.

\begin{prop}\label{prop:lattes-finite-algebraization}

Assume that one of $f_1$ and $f_2$ is a Latt\`es map. Then there exist $N\geq1$ and an
irreducible algebraic curve $C\subset\mathbb P^1\times\mathbb P^1$
such that $F^N(C)=C$ and $\nu(C)>0$.
\end{prop}

\begin{lem}
\label{lem:lattes-propagation}
Let $\nu$ be an ergodic joining of
$(J(f_1),f_1,\mu_1)$ and $(J(f_2),f_2,\mu_2)$. Suppose that $\nu$
charges the graph of a local biholomorphism. If one of $f_1,f_2$ is a Latt\`es map, then so is the other.
\end{lem}

\begin{proof}
By symmetry, suppose that $f_1$ is Latt\`es. Let
$\varphi\colon U\to V$ be a biholomorphic restriction of the
charged local graph and set
$$
\eta:=\nu|_{\Gamma_\varphi},
\qquad
\alpha:=(\pi_1)_*\eta,
\qquad
\beta:=(\pi_2)_*\eta.
$$
Then $\eta\neq0$ and
$
0<\alpha\leq\mu_1,
$
$
0<\beta\leq\mu_2,
$
while
$
\beta=\varphi_*\alpha.
$

Let $m$ denote the spherical area
measure
on $\mathbb P^1$. Since $f_1$ is
Latt\`es, its maximal-entropy measure is absolutely continuous with
respect to $m$.
Hence $\alpha\ll m$. Since $\varphi$ is locally
biholomorphic, it sends $m$-null sets to $m$-null sets locally, and
therefore
$
\beta=\varphi_*\alpha\ll m.
$
Thus $\mu_2$ has a non-zero absolutely continuous component with
respect to $m$.

Write the Lebesgue decomposition
$\mu_2=\mu_2^{\mathrm{ac}}+\mu_2^{\mathrm{s}}$ 
of $\mu_2$
with respect to $m$. Since
$0<\beta\le\mu_2$ and $\beta\ll m$, we have
$\mu_2^{\mathrm{ac}}\ne0$.
Both the image and the inverse image of an $m$-null set under a
rational map are $m$-null. Hence pushforward by $f_2$ preserves both
absolute continuity and singularity with respect to $m$. From
$(f_2)_*\mu_2=\mu_2$ and 
the uniqueness of the Lebesgue decomposition,
it follows that $(f_2)_*\mu_2^{\mathrm{ac}}=\mu_2^{\mathrm{ac}}$
and $(f_2)_*\mu_2^{\mathrm{s}}=\mu_2^{\mathrm{s}}$.

After normalizing $\mu_2^{\mathrm{ac}}$ to have total mass one, we
obtain an $f_2$-invariant probability measure absolutely continuous
with respect to the ergodic invariant probability measure $\mu_2$.
It must therefore equal $\mu_2$, and hence $\mu_2\ll m$. By Zdunik's
characterization \cite{Zdunik90}, $f_2$ is Lattès.
\end{proof}

For the rest of this subsection, assume that both $f_1$ and $f_2$ are
Latt\`es. For $i=1,2$, choose a standard crystallographic
uniformization $\Theta_i\colon\mathbb C\to\mathbb P^1$ with deck
group $\Gamma_i$ and an affine lift
$A_i(z)=a_i z+b_i$, $|a_i|>1$, satisfying
$\Theta_i\circ A_i=f_i\circ\Theta_i$. We choose the presentation so
that
$\Gamma_i=\{z\mapsto\zeta z+\lambda:
\zeta\in\mathcal U_i,\ \lambda\in\Lambda_i\}$, where
$\Lambda_i$ is a rank-two translation lattice and
$\mathcal U_i$ is a finite cyclic group preserving $\Lambda_i$. In particular, $\deg f_i=|a_i|^2$.

Since the branch values of the two uniformizations are finite and the
marginals of $\nu$ are non-atomic, we may restrict the charged graph
to smaller connected discs, still denoted by $U$ and $V$, so that
$\nu(\Gamma_\varphi)>0$, $\varphi\colon U\to V$ is biholomorphic,
and $U$ and $V$ avoid the branch values of $\Theta_1$ and
$\Theta_2$. 

Applying Lemmas \ref{lem:uniform-charged-inverse-box} and
\ref{lem:many-contracting-returns} to this restricted graph, we
retain the notation $D^-\Subset D$, $K$, and $\mathscr B_N$.
We also choose $N$, $W$, and a repelling fixed point $p$ as in
Lemma \ref{lem:charged-graph-return-configuration}.

\begin{lem}\label{lem:lattes-local-affine}
Let $D$ be the return disc from Lemma
\ref{lem:charged-graph-return-configuration}, and choose components
$\widetilde D\subset\Theta_1^{-1}(D)$ and
$\widetilde V\subset\Theta_2^{-1}(\varphi(D))$. Then the lift
$H:=(\Theta_2|_{\widetilde V})^{-1}\circ\varphi\circ
(\Theta_1|_{\widetilde D})$ is the restriction of an affine map
$H(z)=cz+e$ with $c\neq0$.
\end{lem}

\begin{proof}
Let $p$ be the repelling fixed point from Lemma
\ref{lem:charged-graph-return-configuration} and set $q:=\varphi(p)$. 
The relation $\varphi\circ f_1^N=f_2^N\circ\varphi$ near $p$, together
with $f_1^N(p)=p$, gives $f_2^N(q)=q$.
Choose lifts
$\widetilde p\in\widetilde D$ and
$\widetilde q:=H(\widetilde p)$. Since $p$ and $q$ are fixed by
$f_1^N$ and $f_2^N$, respectively, there are deck transformations
$\gamma_i\in\Gamma_i$ such that
$A_1^N(\widetilde p)=\gamma_1(\widetilde p)$ and
$A_2^N(\widetilde q)=\gamma_2(\widetilde q)$.
Set $B_i:=\gamma_i^{-1}\circ A_i^N$. Then $B_1$ fixes
$\widetilde p$ and $B_2$ fixes $\widetilde q$. The local relation
$\varphi\circ f_1^N=f_2^N\circ\varphi$ on $W$ lifts near
$\widetilde p$ to $H\circ B_1=B_2\circ H$.
Indeed, both sides are lifts through $\Theta_2$ of the same local
map and agree at $\widetilde p$.

Translate source and target coordinates so that
$\widetilde p=\widetilde q=0$. Since the $B_i$ are affine and fix the
origin, write $B_i(z)=\lambda_i z$. The linear part of a deck
transformation has modulus one, so
$|\lambda_i|=|a_i|^N>1$. The lifted relation becomes
$H(\lambda_1z)=\lambda_2H(z)$. Since $H'(0)\neq0$, differentiation at
zero gives $\lambda_1=\lambda_2=: \lambda$. Writing
$H(z)=\sum_{k\geq1}c_kz^k$, we obtain
$c_k(\lambda^k-\lambda)=0$ for every $k\geq2$. Since
$|\lambda|>1$, all $c_k$ with $k\geq2$ vanish. Hence $H$ is affine
near the origin, and the identity theorem gives
$H(z)=cz+e$ on the connected domain $\widetilde D$.
\end{proof}

\begin{lem}\label{lem:lattes-return-parameter-bound}
There is a constant $C>0$, independent of the return time $N$, such
that every returning branch $\alpha\in\mathscr B_N$ has a canonical
lift $\widetilde\alpha=A_1^{-N}\circ\gamma_\alpha$, where
$\gamma_\alpha(z)=\zeta_\alpha z+\lambda_\alpha$ with
$\zeta_\alpha\in\mathcal U_1$,
$\lambda_\alpha\in\Lambda_1$, and
$|\lambda_\alpha|\leq C|a_1|^N$. Distinct returning branches give
distinct $\gamma_\alpha$.
\end{lem}

\begin{proof}
Since $D$ avoids the branch values of $\Theta_1$, the restriction
$\Theta_1|_{\widetilde D}\colon\widetilde D\to D$ is biholomorphic.
For a returning branch $\alpha\colon D\to D$, set
$\widetilde\alpha:=(\Theta_1|_{\widetilde D})^{-1}\circ
\alpha\circ\Theta_1|_{\widetilde D}$.
The identity $f_1^N\circ\alpha=\operatorname{id}_D$ implies
$\Theta_1\circ A_1^N\circ\widetilde\alpha=\Theta_1$. Since
$\widetilde D$ is connected, there is a unique
$\gamma_\alpha\in\Gamma_1$ such that
$A_1^N\circ\widetilde\alpha=\gamma_\alpha$.

Fix $z_0\in\widetilde D$. Since
$\widetilde\alpha(\widetilde D)\subset\widetilde D$, we have
$\gamma_\alpha(z_0)\in A_1^N(\widetilde D)$. The set
$\widetilde D$ is bounded, and if
$A_1^N(z)=a_1^Nz+b_{1,N}$, then
$|b_{1,N}|=O(|a_1|^N)$. Thus
$|\gamma_\alpha(z_0)|\leq C_0|a_1|^N$ for a constant independent of
$N$. Writing
$\gamma_\alpha(z)=\zeta_\alpha z+\lambda_\alpha$ and using the
finiteness of $\mathcal U_1$ gives
$|\lambda_\alpha|\leq C|a_1|^N$. Finally, equality of two
$\gamma_\alpha$'s
implies equality of the corresponding canonical
lifts and hence of the downstairs branches.
\end{proof}

\begin{lem}\label{lem:lattes-original-slope}
There is a finite-index subgroup $\Lambda_1'\subset\Lambda_1$ such
that $c\Lambda_1'\subset\Lambda_2$, where $c$ is the slope of the
affine lift $H(z)=cz+e$.
\end{lem}

\begin{proof}
We first prove an auxiliary finite-index statement for a modified
slope; namely, there exists $c_0\in\mathbb C^*$ such that the subgroup
$\Lambda_1'':=\{\lambda\in\Lambda_1:c_0\lambda\in\Lambda_2\}$
has rank two. In particular, $\Lambda_1''$ has finite index in
$\Lambda_1$.

For arbitrarily large return times $N$, choose
$\mathscr B_N$ as in Lemma \ref{lem:many-contracting-returns}; thus
$|\mathscr B_N|\geq\delta d_1^N=\delta|a_1|^{2N}$. For
$\alpha\in\mathscr B_N$, write
$\gamma_\alpha(z)=\zeta_\alpha z+\lambda_\alpha$ as in Lemma
\ref{lem:lattes-return-parameter-bound}. Since $\mathcal U_1$ is
finite, for each such $N$ there exist
a subcollection
$\mathscr B_N'\subset\mathscr B_N$ and
$\zeta_N\in\mathcal U_1$ such that
$|\mathscr B_N'|\geq\delta'|a_1|^{2N}$ and
$\zeta_\alpha=\zeta_N$ for every
$\alpha\in\mathscr B_N'$, where $\delta'>0$ is independent of $N$.

For $\alpha\in\mathscr B_N'$, the identity
$f_2^N\circ\varphi\circ\alpha=\varphi$ lifts to
$A_2^N\circ H\circ A_1^{-N}\circ\gamma_\alpha
=\delta_\alpha\circ H$ for some
$\delta_\alpha\in\Gamma_2$. Write
$\delta_\alpha(z)=\xi_\alpha z+\mu_\alpha$, with
$\xi_\alpha\in\mathcal U_2$ and $\mu_\alpha\in\Lambda_2$.
Comparing linear parts and using $H(z)=cz+e$ gives
$(a_2^N/a_1^N)\zeta_N=\xi_\alpha$. Thus $\xi_\alpha$ is independent
of $\alpha$; denote it by $\xi_N$.

The pairs $(\zeta_N,\xi_N)$ lie in the finite set
$\mathcal U_1\times\mathcal U_2$. Passing to a sequence of return
times $N_k\to\infty$, we may therefore assume that
$\zeta_{N_k}=\zeta$ and $\xi_{N_k}=\xi$ are fixed. Set
$c_0:=c\xi\zeta^{-1}$ and
$\Lambda_1'':=\{\lambda\in\Lambda_1:c_0\lambda\in\Lambda_2\}$.
This is now a single subgroup, independent of $k$.

For every $k$ and
$\alpha,\beta\in\mathscr B_{N_k}'$, subtracting the constant terms
in the two lifted identities gives
$c_0(\lambda_\alpha-\lambda_\beta)
=\mu_\alpha-\mu_\beta\in\Lambda_2$. Hence the parameters
$\lambda_\alpha$, $\alpha\in\mathscr B_{N_k}'$, lie in one translate
of the fixed subgroup $\Lambda_1''$. They are distinct and, by Lemma
\ref{lem:lattes-return-parameter-bound}, satisfy
$|\lambda_\alpha|\leq C|a_1|^{N_k}$.

If $\Lambda_1''$ has rank zero, each translate contains only one
point, which is impossible for large $k$. Suppose that
$\Lambda_1''$ has rank one. Since it is a fixed discrete rank-one
subgroup of $\mathbb C$, there is $C_1>0$ such that every translate
of $\Lambda_1''$ contains at most $C_1(1+R)$ points in $D(0,R)$.
Therefore
$|\mathscr B_{N_k}'|\leq C_2|a_1|^{N_k}$ for a constant $C_2$
independent of $k$. This contradicts
$|\mathscr B_{N_k}'|\geq\delta'|a_1|^{2N_k}$ as $k\to\infty$.
Thus $\Lambda_1''$ has rank two, and hence finite index in
$\Lambda_1$.

\medskip

By the standard crystallographic presentation,
$\mathcal U_i$ preserves $\Lambda_i$. Hence
$\zeta\Lambda_1=\Lambda_1$ and
$\xi\Lambda_2=\Lambda_2$. Set
$\Lambda_1':=\zeta^{-1}\Lambda_1''$. This subgroup has finite index
in $\Lambda_1$. If $\lambda''=\zeta^{-1}\lambda$ with
$\lambda\in\Lambda_1''$, then
$\xi\zeta^{-1}c\lambda\in\Lambda_2$. Applying $\xi^{-1}$ and using
commutativity of multiplication in $\mathbb C$ gives
$c\lambda''\in\Lambda_2$.
\end{proof}

\begin{proof}[Proof of Proposition
\ref{prop:lattes-finite-algebraization}]
By Lemma \ref{lem:lattes-propagation}, both $f_1$ and $f_2$ are
Latt\`es. Lemma \ref{lem:lattes-local-affine} gives an affine lift
$H(z)=cz+e$ of the charged local graph. Lemma
\ref{lem:lattes-original-slope} gives a finite-index subgroup
$\Lambda_1'\subset\Lambda_1$ such that
$c\Lambda_1'\subset\Lambda_2$. 

Now we show that the charged local graph $\Gamma_{\varphi|_D}$ is contained in an
irreducible algebraic curve
$C\subset\mathbb P^1\times\mathbb P^1$, and $\nu(C)>0$. Set $E_i:=\mathbb C/\Lambda_i$, let
$p_i\colon\mathbb C\to E_i$ be the quotient map, and write
$\Theta_i=\pi_i\circ p_i$ with
$\pi_i\colon E_i\to\mathbb P^1$ finite. Set
$E_1':=\mathbb C/\Lambda_1'$. The inclusion
$\Lambda_1'\subset\Lambda_1$ induces an isogeny
$q\colon E_1'\to E_1$, while $H(z)=cz+e$ descends, by Lemma
\ref{lem:lattes-original-slope}, to a holomorphic map
$h\colon E_1'\to E_2$. Define
$\Xi:=(\pi_1\circ q,\pi_2\circ h)\colon
E_1'\to\mathbb P^1\times\mathbb P^1$ and set
$C:=\Xi(E_1')$. Since $E_1'$ is a compact irreducible
algebraic curve, $C$ is an irreducible algebraic curve.
On $\widetilde D$ we have
$\Theta_2\circ H=\varphi\circ\Theta_1$, so $C$ contains the graph of
$\varphi$ over $D$. As in the proof of Proposition
\ref{prop:finite-case-algebraization}, the positive-measure set $K$
from Lemma \ref{lem:uniform-charged-inverse-box} projects into this
graph. Hence $\nu(C)>0$.

It remains to show that there exists a return time $N\ge1$ such that $F^N(C)=C$. By construction, $C$ contains $\Gamma_{\varphi|_D}$. Choose a sufficiently large return time $N$ and a corresponding
branch $\alpha\in\mathscr B_N$ from Lemma
\ref{lem:many-contracting-returns}.
Then
$\alpha(D)\Subset D$, $f_1^N\circ\alpha=\operatorname{id}_D$, and
$f_2^N\circ\varphi\circ\alpha=\varphi$ on $D$.

For every $x\in D$, the point
$(\alpha(x),\varphi(\alpha(x)))$ belongs to $C$, and its image under
$F^N$ is $(x,\varphi(x))$. Hence
$\Gamma_{\varphi|_D}\subset F^N(C)$. The image $F^N(C)$ is an
irreducible algebraic curve. Since $C$ and $F^N(C)$ contain the same
non-empty analytic graph, they coincide. Thus $F^N(C)=C$.
\end{proof}

\subsection{Classification and conclusion of the proof of Theorem \ref{thm:Cnu-finite}}
\label{subsec:finite-classification}

\begin{prop}\label{prop:finite-case-classification}
Let $C\subset\mathbb P^1\times\mathbb P^1$ be an irreducible
algebraic curve with $\nu(C)>0$ and $F^N(C)=C$ for some $N\geq1$.
Then exactly one of alternatives \textup{(A1)} and \textup{(A2)}
holds.
\end{prop}

\begin{proof}
Both coordinate projections of $C$ are dominant. Indeed, if
$\pi_1(C)$ were a point, then $C$ would be vertical and
$\nu(C)\leq\mu_1(\{a\})=0$, contradicting $\nu(C)>0$; the same
argument applies to $\pi_2$.

If both projection degrees are greater than one, $C$ gives
alternative \textup{(A2)}. 
We now assume that at least one
projection has degree one.

Let $r\ge1$ be minimal with $F^r(C)=C$, and set
$C_j:=F^j(C)$ for $0\le j<r$. The curves $C_j$ are distinct and
$F$ cyclically permutes them. Set $\mathcal U:=\bigcup_{j=0}^{r-1}C_j$.
Since $F(\mathcal U)=\mathcal U$ and $\nu(\mathcal U)\ge\nu(C)>0$,
invariance and ergodicity of $\nu$ give $\nu(\mathcal U)=1$.

The joining $\nu$ is non-atomic as its marginals are
non-atomic. Since distinct irreducible curves meet in finitely many
points, the intersections $C_j\cap C_k$, for $j\ne k$, are
$\nu$-null. Since $\nu(\mathcal U)=1$ and, modulo the finite intersections of
distinct curves, the only component of $\mathcal U$ mapped into
$C_{j+1}$ is $C_j$, invariance of $\nu$ gives
$F_*(\nu|_{C_j})=\nu|_{C_{j+1\!\!\!\pmod r}}$. Consequently,
$\nu(C_j)=1/r$ for every $j$.

If some $C_j$ has both projection degrees greater than one, then
$F^r(C_j)=C_j$ and alternative (A2) holds. Otherwise, for every $j$,
at least one coordinate projection of $C_j$ has degree one. Hence
$C_j$ is either the graph of a rational map or the transpose of such
a graph.

Set $\nu_j:=r\,\nu|_{C_j}$. Then each $\nu_j$ is a probability
measure, $F_*\nu_j=\nu_{j+1\!\!\!\pmod r}$, and
$(F^r)_*\nu_j=\nu_j$. Suppose that $C_j=\Gamma_{A_j}$. The first
marginal of $\nu_j$ is an $f_1^r$-invariant probability measure
absolutely continuous with respect to $\mu_1$. Since $\mu_1$ is
mixing, and hence ergodic for $f_1^r$, this marginal equals $\mu_1$.
Therefore $\nu_j=(\operatorname{id},A_j)_*\mu_1$. Moreover,
$F^r(C_j)=C_j$ gives
$f_2^r\circ A_j=A_j\circ f_1^r$.

If $C_j$ is the transpose of the graph of a rational map $B_j$, the
same argument using the second marginal gives
$\nu_j=(B_j,\operatorname{id})_*\mu_2$ and
$f_1^r\circ B_j=B_j\circ f_2^r$. Finally,
$\nu=\frac1r\sum_{j=0}^{r-1}\nu_j$, so these measures give precisely
the cyclic decomposition in alternative (A1).

Finally, alternatives \textup{(A1)} and \textup{(A2)} are mutually
exclusive. Indeed, in \textup{(A1)} the joining is supported on a
finite union of graph or transpose-graph curves. An irreducible curve
as in \textup{(A2)} has finite intersection with each such component
unless they coincide, while coincidence is impossible as both
projection degrees in \textup{(A2)} are greater than one. Since
$\nu$ is non-atomic, such a curve cannot have positive $\nu$-measure.
\end{proof}

\begin{proof}[Proof of Theorem \ref{thm:Cnu-finite}]
If neither $f_1$ nor $f_2$ is Latt\`es, Proposition
\ref{prop:finite-case-algebraization} gives an invariant irreducible
algebraic curve of positive joining measure. If one of $f_1,f_2$ is
Latt\`es, Proposition \ref{prop:lattes-finite-algebraization} gives
such a curve. In either case, Proposition
\ref{prop:finite-case-classification} gives exactly one of
alternatives \textup{(A1)} and \textup{(A2)}.
\end{proof}

\section{Joinings without charged local holomorphic graphs}
\label{sec:Cnu-infinite}
We now treat joinings for which no local biholomorphic graph has
positive joining measure. The main result is the following theorem.

\begin{thm}\label{thm:Cnu-infinite}
Let $\nu$ be an ergodic joining of $f_1$ and $f_2$. Suppose that
$\nu(\Gamma_\varphi)=0$ for every biholomorphism
$\varphi\colon U'\to V'$ between non-empty open subsets of
$\mathbb P^1$. Then $\mathcal C_\nu$ is infinite, and
the family
$$
\mathscr H_\nu
:=
\left\{
\chi_b^{-1}
\circ\widetilde\Psi\circ
\chi_a
:
\widetilde\Psi\in\overline{\mathcal C_\nu}
\right\}
\subset\operatorname{Hol}(U,V)
$$
is compact and non-discrete. Every element of $\mathscr H_\nu$ is
non-constant and is locally biholomorphic after restriction to a
non-empty open subset.
\end{thm}

\begin{rmk}\label{rmk:product-A2}
Theorem \ref{thm:Cnu-infinite} also applies to the product joining,
which is mixing and hence ergodic. Indeed, since $\mu_2$ is
non-atomic, Fubini's theorem shows that
$(\mu_1\times\mu_2)(\Gamma_\varphi)=0$ for every local measurable
map $\varphi$. Thus the product joining satisfies the hypothesis of
Theorem \ref{thm:Cnu-infinite}, which produces a compact
non-discrete family of local holomorphic maps.

For the product joining this phenomenon is automatic from the
transfer-map construction and does not reflect any dependence between
the two coordinates. For the non-product joinings considered in
Theorem \ref{thm_main}, the reference set $E^\sharp$ can be chosen
to satisfy \eqref{eq:25-extra}, so that the reference data themselves
witness a genuine dependence between the two coordinates.
\end{rmk}

\begin{lem}\label{lem:Cnu-closure-nonconstant}
Assume that $\mathcal C_\nu$ is infinite. Then
$\overline{\mathcal C_\nu}\subset\mathcal K$
is compact and non-discrete. Moreover, every
$\widetilde\Psi\in\overline{\mathcal C_\nu}$ is a non-constant
holomorphic map
$
\widetilde\Psi\colon
D(0,r^\sharp)\to D(0,R^\sharp),
$
and is locally biholomorphic after restriction to a non-empty open
subset of $\chi_a(U)$.
\end{lem}

\begin{proof}
Since $\mathcal K$ is compact and metrizable,
$\overline{\mathcal C_\nu}$ is compact. As $\mathcal C_\nu$ is
infinite, its closure has an accumulation point and is therefore
non-discrete.
It remains to prove the assertions about the elements of
$\overline{\mathcal C_\nu}$. Let
$\widetilde\Psi_j\in\mathcal C_\nu$ converge locally uniformly on
$D(0,r^\sharp)$ to
$\widetilde\Psi\in\overline{\mathcal C_\nu}$.
For each $j$, choose
$z_j=(\hat x_j,\hat y_j)\in E^\sharp_{\rm rec}$ such that
$\widetilde\Psi_j\in C(z_j)$, and set
$q_j:=\chi_a(x_{j,0})$.
By Proposition \ref{prop:reference-charts} (1), the points $q_j$
belong to the compact set
$\overline{\chi_a(U)}\Subset D(0,r^\sharp)$.
After passing to a subsequence, we may assume that
$q_j\to q\in\overline{\chi_a(U)}$.

Since $\widetilde\Psi_j\in C(z_j)$, Proposition
\ref{prop:reference-charts} (4) gives
$|\widetilde\Psi_j'(q_j)|\geq c_*$ for every $j$.
Local uniform convergence of holomorphic maps implies local uniform
convergence of their derivatives on compact subsets. Hence
$\widetilde\Psi_j'(q_j)\to\widetilde\Psi'(q)$, and therefore
$
|\widetilde\Psi'(q)|\geq c_*>0.
$
Thus $\widetilde\Psi$ is non-constant. By the inverse function
theorem, $\widetilde\Psi$ is biholomorphic on a neighbourhood $W$ of
$q$. Since $q\in\overline{\chi_a(U)}$, the open set
$W\cap\chi_a(U)$ is non-empty, so $\widetilde\Psi$ is locally
biholomorphic on a non-empty open subset of $\chi_a(U)$.

Finally, every $\widetilde\Psi_j$ takes values in
$D(0,R^\sharp)$. Hence the locally uniform limit
$\widetilde\Psi$ takes values in
$\overline{D(0,R^\sharp)}$. Since $\widetilde\Psi$ is non-constant,
the maximum principle shows that
$
\widetilde\Psi(D(0,r^\sharp))\subset D(0,R^\sharp).
$
This proves the lemma.
\end{proof}

\begin{proof}[Proof of Theorem \ref{thm:Cnu-infinite}]
Suppose, toward a contradiction, that $\mathcal C_\nu$ is finite.
By Lemma \ref{lem:Cnu-finite-admissible}, there exists
$\widetilde\Psi\in\mathcal A_\nu$. Let
$\varphi:=\varphi_{\widetilde\Psi}$ be the corresponding holomorphic
map given by Proposition
\ref{prop:realization-admissible-germs}.

For every
$z=(\hat x,\hat y)\in E^\sharp_{\widetilde\Psi}$,
Proposition \ref{prop:realization-admissible-germs} gives
$\varphi'(x_0)\neq0$. Hence, by the inverse function theorem, there
is an open neighbourhood $U_z$ of $x_0$ on which $\varphi$ is
biholomorphic onto its image. By second countability, there are
countably many such open sets $U_j, j\geq 1$,
whose union contains $x_0$ for every
$z=(\hat x,\hat y)\in E^\sharp_{\widetilde\Psi}$. Therefore
$$
E^\sharp_{\widetilde\Psi}
\subset
\bigcup_{j\geq1}
\pi^{-1}\bigl(\Gamma_{\varphi|_{U_j}}\bigr).
$$
Since
$\hat\nu(E^\sharp_{\widetilde\Psi})>0$, there exists $j$ such that
$
\hat\nu\!\left(
\pi^{-1}\bigl(\Gamma_{\varphi|_{U_j}}\bigr)
\right)>0.
$
As $\pi_*\hat\nu=\nu$, this gives
$\nu(\Gamma_{\varphi|_{U_j}})>0$, contradicting the hypothesis that
$\nu$ charges no graph of a local biholomorphism. Hence
$\mathcal C_\nu$ is infinite.
By Lemma
\ref{lem:Cnu-closure-nonconstant}, the closure
$\overline{\mathcal C_\nu}
\subset\mathcal K$
is compact and non-discrete. Moreover, every
$\widetilde\Psi\in\overline{\mathcal C_\nu}$
is non-constant and is
locally biholomorphic after restriction to a non-empty open subset 
of $\chi_a(U)$.

For
$\widetilde\Psi\in\overline{\mathcal C_\nu}$, 
define
$
\varphi_{\widetilde\Psi}
:= \chi_b^{-1}
\circ\widetilde\Psi\circ
\chi_a$.
By Proposition \ref{prop:reference-charts}, these maps are defined on
the fixed reference neighbourhood $U$ and take values in $V$.
Hence 
$\mathscr H_\nu=
\{
\varphi_{\widetilde\Psi}:
\widetilde\Psi\in\overline{\mathcal C_\nu}
\}
\subset\operatorname{Hol}(U,V)$. Consider
the map
$$
\Theta\colon
\overline{\mathcal C_\nu}
\to
\operatorname{Hol}(U,V),
\qquad
\Theta(\widetilde\Psi)
= \varphi_{\widetilde \Psi}
=
\chi_b^{-1}
\circ\widetilde\Psi\circ
\chi_a.
$$
Here $\widetilde\Psi$ is restricted to the relatively compact domain
$\chi_a(U)\Subset D(0,r^\sharp)$. The map $\Theta$ is continuous for
the compact-open topologies. It is also injective: if
$\Theta(\widetilde\Psi_1)=\Theta(\widetilde\Psi_2)$ on $U$, then
$\widetilde\Psi_1=\widetilde\Psi_2$ on the non-empty open set
$\chi_a
(U)$, hence on $D(0,r^\sharp)$ by the identity theorem.

Since $\overline{\mathcal C_\nu}$ is compact and
$\operatorname{Hol}(U,V)$ is Hausdorff, $\Theta$ is a homeomorphism
onto its image. Therefore
$
\mathscr H_\nu=\Theta(\overline{\mathcal C_\nu})
$
is compact and non-discrete.
Moreover, every $h\in\mathscr H_\nu$ is non-constant. Since every
$\widetilde\Psi \in \overline{\mathcal C_\nu}$ is locally biholomorphic
after restriction to a non-empty open subset of
$\chi_a(U)$,
the same is true of the
corresponding map
$h=
\chi_b^{-1}
\circ\widetilde\Psi\circ
\chi_a$.
This completes the proof.
\end{proof}

\section{Proof of Theorem \ref{thm_main}}
\label{sec:pf-main-thm}
We can now complete the proof of Theorem \ref{thm_main}.

\begin{proof}[Proof of Theorem \ref{thm_main}]
If $\nu$ charges the graph of a local biholomorphism, then case
\textup{(A)} holds, and Theorem \ref{thm:Cnu-finite} gives exactly
one of alternatives \textup{(A1)} and \textup{(A2)}. Otherwise case
\textup{(B)} holds, and Theorem \ref{thm:Cnu-infinite} gives the
asserted compact non-discrete family.
\end{proof}

Theorem \ref{thm_main} gives an immediate consequence when the
degrees are different.

\begin{cor}\label{cor:different-degrees}
Assume that $\deg f_1\neq\deg f_2$. Then every non-product ergodic
joining of $f_1$ and $f_2$ belongs to case \textup{(B)} of
Theorem \ref{thm_main}.
\end{cor}

\begin{proof}
Alternative \textup{(A1)} forces
$\deg f_1=\deg f_2$ by taking degrees in
$f_2^r\circ A=A\circ f_1^r$, or in the corresponding transpose-graph
relation. 
For \textup{(A2)}, let $C$ be an invariant curve as in that
alternative, let
$\sigma\colon \widetilde C\to C$ be its normalization, and 
for $i=1,2$, set $p_i:=\pi_i\circ\sigma$.
Since $(f_1^n\times f_2^n)(C)=C$, the restriction of
$f_1^n\times f_2^n$ to $C$ lifts to a self-map
$g\colon \widetilde C\to\widetilde C$ satisfying
$p_1\circ g=f_1^n\circ p_1$ and 
$p_2\circ g=f_2^n\circ p_2$.
Taking degrees gives
$\deg g=(\deg f_1)^n=(\deg f_2)^n$.
Thus, if the degrees are different, the algebraic case
\textup{(A)} cannot occur, and the joining belongs to case
\textup{(B)}.
\end{proof}

The conclusion of Corollary \ref{cor:different-degrees}
is not vacuous; see
Remark \ref{rmk:different-degrees-example}.

\section{Measurable examples}
\label{sec:measurable-examples}
In this section, we give two constructions 
of ergodic joinings in case \textup{(B)} of Theorem \ref{thm_main} with very different
conditional measures. The first has diffuse conditional measures,
while the second has finite atomic conditional measures. 

Denote by
$\Sigma_d:=\{0,\ldots,d-1\}^{\mathbb N}$
the one-sided full shift with the uniform Bernoulli measure
$\beta_d$ and left shift $\sigma$. By the Bernoulli theorem \cite{HeicklenHoffman, ManeBernoulli}, 
for every rational map $f$ of
degree $d$, the system $(J(f),f,\mu_f)$ is measurably isomorphic to
$(\Sigma_d,\sigma,\beta_d)$.

\subsection{Diffuse Bernoulli couplings}
Fix $0<\rho<1$ and give each pair of symbols $(i,j)$ the probability
$$
p_\rho(i,j)
=
\rho d^{-1}\mathbf 1_{\{i=j\}}
+
(1-\rho)d^{-2}.
$$
Taking these pairs independently over time gives an ergodic
non-product self-joining $\lambda_\rho$ of the uniform $d$-shift
with uniform marginals. Indeed, summing $p_\rho(i,j)$ over either coordinate gives the
uniform distribution on the alphabet, so both coordinate marginals
of $\lambda_\rho$ are
equal to $\beta_d$.

The conditional measures of $\lambda_\rho$ over either coordinate
are non-atomic. Indeed, after conditioning on
$x=(x_n)\in\Sigma_d$, the coordinates of the second sequence remain
independent, and every one-symbol conditional probability is at most
$
\rho+\frac{1-\rho}{d}<1.
$
Hence every individual sequence has conditional measure zero.
Consequently, if $E\subset\Sigma_d$ is measurable and
$B\colon E\to\Sigma_d$ is any measurable map, then
$$
\lambda_\rho(\Gamma_B)
=
\int_E
(\lambda_\rho)_x(\{B(x)\})\,d\beta_d(x)
=
0.
$$

Let
$\theta\colon(\Sigma_d,\sigma,\beta_d)\to(J(f),f,\mu_f)$
be a Bernoulli isomorphism and set
$
\nu_\rho:=
(\theta\times\theta)_*\lambda_\rho.
$
Then $\nu_\rho$ is an ergodic non-product self-joining of $f$. We note that since $\theta_*\beta_d=\mu_f$, the two marginals of $\nu_\rho$ are
both $\mu_f$.
Moreover, the conditional measures of $\nu_\rho$ over either
coordinate are, almost everywhere, the pushforwards under $\theta$
of the corresponding conditional measures of $\lambda_\rho$, and
are therefore non-atomic. Hence, if $E\subset J(f)$ is measurable
and $B\colon E\to J(f)$ is any measurable map, then
$
\nu_\rho(\Gamma_B)
=
0.
$
In particular, $\nu_\rho$ charges no graph of a local
biholomorphism and therefore belongs to case \textup{(B)} of
Theorem \ref{thm_main}.

\begin{rmk}\label{rmk:different-degrees-example}
A modification of the above example also shows that Corollary
\ref{cor:different-degrees} is not vacuous. Indeed, suppose that
$d_1\neq d_2$ have a common divisor $c>1$. Identify the uniform
$d_i$-symbol shift with the product of the uniform $c$-symbol shift
and the uniform $(d_i/c)$-symbol shift. Coupling the two common
$c$-symbol coordinates diagonally and taking the two remaining
coordinates independently gives a non-product ergodic joining of
the uniform $d_1$- and $d_2$-symbol shifts. Transporting this joining
through Bernoulli isomorphisms gives a non-product ergodic joining
of any pair of rational maps of degrees $d_1$ and $d_2$, which
belongs to case \textup{(B)} by Corollary
\ref{cor:different-degrees}.
\end{rmk}

\subsection{A finite-to-one measurable joining}
We next show that case \textup{(B)} can also occur for joinings with
finite atomic conditional measures.

By \cite[Theorem 1.2]{PakovichPeriodicCurves}, for every $d\geq2$
there exists a non-empty Zariski open subset
$\mathcal U\subset\operatorname{Rat}_d$ such that, for
$f_1,f_2\in\mathcal U$, the product map $f_1\times f_2$ admits an
irreducible periodic curve which is neither vertical nor horizontal
if and only if $f_1$ and $f_2$ are M\"obius-conjugate.
Choose non-conjugate maps $f_1,f_2\in\mathcal U$.
Let
$
\theta_i\colon
(\Sigma_d,\sigma,\beta_d)
\to
(J(f_i),f_i,\mu_i),
 i=1,2,
$
be Bernoulli isomorphisms. Identify the alphabet with
$\mathbb Z/d\mathbb Z$ and define
$T\colon\Sigma_d\to\Sigma_d$ by
$$
T(x)_n=x_n+x_{n+1}\pmod d.
$$
Then $T\circ\sigma=\sigma\circ T$. Moreover,
$T_*\beta_d=\beta_d$. Indeed, for every block
$(y_0,\ldots,y_{N-1})$, choosing $x_0$ determines recursively a
unique block $(x_0,\ldots,x_N)$ satisfying
$x_n+x_{n+1}=y_n$. Hence there are exactly $d$ preimage blocks, and
the corresponding cylinder has probability $d^{-N}$.

Set
$
\lambda_T:=
(\operatorname{id},T)_*\beta_d
$
and
$
\nu_T:=
(\theta_1\times\theta_2)_*\lambda_T$.
Since
$(\pi_1)_*\lambda_T=\beta_d$ and
$(\pi_2)_*\lambda_T=T_*\beta_d=\beta_d$, the two marginals of
$\nu_T$ are $\mu_1$ and $\mu_2$. Hence $\nu_T$ is a joining of
$f_1$ and $f_2$. It is ergodic as $\lambda_T$ is the image of
the ergodic system $(\Sigma_d,\sigma,\beta_d)$ under the equivariant
map $x\mapsto(x,T(x))$. It is non-product since its conditional measure over the first
coordinate is a Dirac mass, whereas the corresponding conditional
measure for $\mu_1\times\mu_2$ is the non-atomic measure $\mu_2$.

Moreover, $\nu_T$ has finite atomic conditional measures over the
second coordinate. Indeed, $\lambda_T$ is supported on the graph of
$T$, and every fiber of $T$ consists of exactly $d$ points. Thus the
conditional measures of $\lambda_T$ over the second coordinate, and
hence those of $\nu_T$, are supported on at most $d$ points almost
everywhere.

We claim that $\nu_T$ belongs to case \textup{(B)} of
Theorem \ref{thm_main}. Suppose not. Then $\nu_T$ belongs to
case \textup{(A)}, and Theorem \ref{thm:Cnu-finite} gives an integer
$N\geq1$ and an irreducible algebraic curve
$C\subset\mathbb P^1\times\mathbb P^1$ such that
$(f_1^N\times f_2^N)(C)=C$
and $\nu_T(C)>0$.
Since the marginals of $\nu_T$ are the non-atomic measures
$\mu_1$ and $\mu_2$, the curve $C$ cannot be vertical or horizontal.
Thus $C$ is a non-vertical, non-horizontal periodic curve for
$f_1\times f_2$. This contradicts
\cite[Theorem 1.2]{PakovichPeriodicCurves}, since $f_1$ and $f_2$
were chosen to be non-conjugate. Hence $\nu_T$ belongs to case
\textup{(B)}.

\section{Recurrence and cluster structure modulo gauge}
\label{sec:cluster-gauge}
The proof of Theorem \ref{thm_main} uses only the total cluster
family $\mathcal C_\nu$. The fiberwise cluster sets $C(z)$,
however, contain a finer structure coming from recurrence. The
purpose of this section is to show that recurrent returns preserve
this cluster structure up to natural changes of coordinates.

\begin{thm}\label{thm:cluster-gauge-equivariance}
For every $z\in E^\sharp_{\mathrm{rec}}$, the fiberwise cluster sets
$C(z)$ and $C(Tz)$ are gauge equivalent. Equivalently, we have
$
[C(Tz)]_{\rm gau}
=
[C(z)]_{\rm gau}.
$
\end{thm}

In Section \ref{subsec:return-gauge}, we compare the normalized
transfer maps based at $z$ and at its recurrent return $Tz$ and
identify the resulting left and right coordinate changes.
In Section \ref{subsec:gauge-equivariance}, we use these changes to
define the
gauge equivalence
and prove Theorem
\ref{thm:cluster-gauge-equivariance}.

\subsection{Return maps and coordinate changes}
\label{subsec:return-gauge}
Recall that
$E^\sharp_{\mathrm{rec}}\subset E^\sharp$ is the recurrent part for
$\hat F^{-1}$, that
$\tau\colon E^\sharp_{\mathrm{rec}}\to\mathbb N$ is the first return
time to $E^\sharp$, and that
$Tz=\hat F^{-\tau(z)}z$ is the induced return map.
The proof of Theorem \ref{thm:cluster-gauge-equivariance} relies on
the following comparison
between the normalized transfer maps based at $z$ and at $Tz$.

\begin{lem}\label{lem:return-equivar}
Fix $z=(\hat x,\hat y)\in E^\sharp_{\mathrm{rec}}$, write
$Tz=(\hat x^+,\hat y^+)$, and set $\tau:=\tau(z)$. Then, after
restricting to sufficiently small neighbourhoods depending on $z$,
there exist a local biholomorphic germ $B_z$ and local
biholomorphic germs $A_{z,n}$ such that, for all sufficiently large
$n$, we have
\begin{equation}\label{eq:tilde-gauge}
\widetilde\Phi_{z,n+\tau}
=
A_{z,n}\circ
\widetilde\Phi_{Tz,n}\circ B_z.
\end{equation}
Moreover, every sequence $n_j\to\infty$ admits a subsequence along
which $A_{z,n_j}$ converges locally uniformly to a locally
biholomorphic germ.
\end{lem}

\begin{proof}
Since $Tz=\hat F^{-\tau}z$, we have
$\hat x^+=\hat f_1^{-\tau}\hat x$ and
$\hat y^+=\hat f_2^{-\tau}\hat y$. The inverse branches satisfy
$$
G_{1,\hat x,n+\tau}
=
G_{1,\hat x^+,n}\circ G_{1,\hat x,\tau}
\qquad
\text{and}
\qquad
G_{2,\hat y,n+\tau}
=
G_{2,\hat y^+,n}\circ G_{2,\hat y,\tau}.
$$
Recall that
$a_n(z)=G'_{2,\hat y,n}(0)/G'_{1,\hat x,n}(0)$. Differentiating the
preceding identities gives
$$
a_{n+\tau}(z)
=
a_n(Tz)\lambda_z
\qquad
\text{and}
\qquad
\lambda_z
:=
\frac{G'_{2,\hat y,\tau}(0)}
     {G'_{1,\hat x,\tau}(0)}
\in\mathbb C^*.
$$
Using the definitions of the normalized transfer maps and the cocycle
relations, we obtain
$$
\Phi_{z,n+\tau}
=
L_{z,n}\circ
\Phi_{Tz,n}\circ
G_{1,\hat x,\tau},
\quad \text{where}
\quad
L_{z,n}
:=
G_{2,\hat y,\tau}^{-1}
\circ
G_{2,\hat y^+,n}^{-1}
\circ
(\lambda_z\cdot\operatorname{id})
\circ
G_{2,\hat y^+,n}.
$$
All these identities are understood on sufficiently small
neighbourhoods on which the compositions are defined.
Passing to the fixed reference charts gives
\eqref{eq:tilde-gauge}, with
$$
A_{z,n}
:=
k_{\hat y}\circ L_{z,n}\circ k_{\hat y^+}^{-1}
\qquad
\text{and}
\qquad
B_z
:=
h_{\hat x^+}^{-1}\circ
G_{1,\hat x,\tau}\circ h_{\hat x}.
$$
The germ $B_z$ is a local biholomorphism depending only on the finite
return segment.
For fixed $z$, the only $n$-dependent part of $A_{z,n}$ is 
$$
H_n
:=
G_{2,\hat y^+,n}^{-1}
\circ
(\lambda_z\cdot\operatorname{id})
\circ
G_{2,\hat y^+,n}.
$$
By Proposition \ref{prop:inverse branch-package}, the maps
$G_{2,\hat y^+,n}$ are univalent on a common disc and satisfy uniform
Koebe estimates there. Choose $t>0$ sufficiently small so that
$G_{2,\hat y,\tau}^{-1}$ is defined on $D(0,t)$. Since $\lambda_z$
is fixed, we may then choose $0<s<t$, depending on $z$ but not on
$n$, such that
$$
\lambda_z G_{2,\hat y^+,n}(D(0,s))
\subset
G_{2,\hat y^+,n}(D(0,t))
$$
for every $n$. Indeed, the Koebe distortion and growth estimates give
uniform inner and outer bounds for these images in terms of
$|G'_{2,\hat y^+,n}(0)|$.
It follows that $H_n$ is defined on $D(0,s)$ for every $n$ and
satisfies
$H_n(D(0,s))\subset D(0,t)$.
Hence $\{H_n\}$ is a normal family on $D(0,s)$.
Moreover,
$H_n'(0)=\lambda_z$
for every $n$.
Consequently, every locally uniform cluster limit of the $H_n$ has
non-zero derivative at the origin and is therefore locally
biholomorphic after restriction.

Since $A_{z,n}$ is obtained from $H_n$ by pre- and
post-composition with fixed local biholomorphisms,
after restricting
to a sufficiently small fixed neighbourhood of the base point, the
normality of $\{H_n\}$ gives normality of
$\{A_{z,n}\}$. Every cluster limit is locally biholomorphic as
the same holds for the corresponding cluster limit of $\{H_n\}$.
\end{proof}

\subsection{Gauge equivariance of fiberwise cluster sets}
\label{subsec:gauge-equivariance}
We now formalize the coordinate ambiguity appearing in Lemma
\ref{lem:return-equivar}. Recall from Proposition
\ref{prop:reference-charts} (4) that, for
$z=(\hat x,\hat y)\in E^\sharp_{\mathrm{rec}}$, we set
$q_z:=\chi_a(x_0)$. When an element of $C(z)$ is regarded as a germ,
we mean its germ at $q_z$.
We denote by $\mathcal G_R$
the local pseudo-group generated by the right
return germs $B_z$,
$z\in E^\sharp_{\mathrm{rec}}$,
together with their restrictions, local inverses, and compositions
whenever defined. Similarly, 
we denote by $\mathcal G_L$
the local pseudo-group generated by the germs
$A_{z,n}$, for $z\in E^\sharp_{\mathrm{rec}}$ and all sufficiently
large $n$, and by all compact-open limits of sequences
$A_{z,n_j}$ with $z$ fixed and $n_j\to\infty$, together with their
restrictions, local inverses, and compositions whenever defined.

\begin{defn}\label{def:gauge-pgroup}
Two local holomorphic germs $\Psi_1$ and $\Psi_2$ in the fixed
reference coordinates are \emph{gauge equivalent} if, after
restriction to sufficiently small neighbourhoods, we have
$\Psi_2=A\circ\Psi_1\circ B$
for some $A\in\mathcal G_L$ and $B\in\mathcal G_R$.
Two families $K_1,K_2$ of such germs are
\emph{gauge equivalent} if every element of $K_1$ is gauge equivalent
to an element of $K_2$, and conversely. We write
$[K]_{\rm gau}$ for the resulting equivalence class.
\end{defn}

We can now prove Theorem \ref{thm:cluster-gauge-equivariance}.

\begin{proof}[Proof of Theorem \ref{thm:cluster-gauge-equivariance}]
Fix $\widetilde\Psi\in C(Tz)$.
Choose $n_j\to\infty$ such that
$\widetilde\Phi_{Tz,n_j}\to\widetilde\Psi$ locally uniformly. By
Lemma \ref{lem:return-equivar}, after passing to a subsequence we may
assume that
$A_{z,n_j}\to A$ locally uniformly for a locally biholomorphic germ
$A$. Moreover, we have
$$
\widetilde\Phi_{z,n_j+\tau(z)}
=
A_{z,n_j}\circ
\widetilde\Phi_{Tz,n_j}\circ B_z
$$
as germs on a fixed neighbourhood.

By normality of the family
$\{\widetilde\Phi_{z,n}\}_{n\geq1}$, after passing to a further
subsequence we may assume that
$
\widetilde\Phi_{z,n_j+\tau(z)}
\to\widetilde\Theta
$
locally uniformly on $D(0,r^\sharp)$ for some
$\widetilde\Theta\in C(z)$. Passing to the limit in the preceding
identity gives, as germs,
$\widetilde\Theta
=
A\circ\widetilde\Psi\circ B_z$.
Thus every element of $C(Tz)$ is gauge equivalent to an element of
$C(z)$.

Conversely, take $\widetilde\Theta\in C(z)$ and choose
$m_j\to\infty$ such that
$\widetilde\Phi_{z,m_j}\to\widetilde\Theta$ locally uniformly.
After discarding finitely many terms, write
$m_j=n_j+\tau(z)$ with $n_j\to\infty$.
By Lemma \ref{lem:return-equivar}, after passing to a subsequence we
may assume that $A_{z,n_j}\to A$ locally uniformly. Normality of the
transfer maps also allows us to assume that
$
\widetilde\Phi_{Tz,n_j}
\to
\widetilde\Psi
$
for some $\widetilde\Psi\in C(Tz)$. Passing to the limit in
$$
\widetilde\Phi_{z,m_j}
=
A_{z,n_j}\circ
\widetilde\Phi_{Tz,n_j}\circ B_z
$$
gives again, as germs,
$
\widetilde\Theta
=
A\circ\widetilde\Psi\circ B_z.
$
This completes the proof.
\end{proof}

\begin{rmk}\label{rmk:gauge-empty-Anu}
Theorem \ref{thm:cluster-gauge-equivariance} is especially useful
when $\mathcal A_\nu=\emptyset$. In that case no individual cluster
germ occurs on a set of positive lifted joining measure, but recurrent
returns still transport the fiberwise cluster sets into one another
modulo gauge.
\end{rmk}

\section{Classical exceptional maps are joining-exceptional}
\label{sec:classical-ex-maps}
In this section, we specialize to the diagonal case
$f_1=f_2=f$ and show that the classical exceptional rational maps
are joining-exceptional. Recall that {\it classical
exceptional maps} are
rational maps that are M\"obius-conjugate to a power
map $z\mapsto z^{\pm d}$, to a signed Chebyshev map $\pm T_d$,
where $d\geq2$ and $T_d$ is normalized by
$T_d(z+z^{-1})=z^d+z^{-d}$,
or to a Latt\`es map
(i.e.,
a finite quotient of an affine
endomorphism of an elliptic curve).

\begin{prop}\label{prop:def-dim-1}
Every classical exceptional rational map of degree at least $2$ is
joining-exceptional.
\end{prop}

\begin{proof}
It suffices to prove the statement for the models $z^d$, $T_d$, and
the Latt\`es maps. Indeed, joining-exceptionality is invariant under
M\"obius conjugacy: if $g=M\circ f\circ M^{-1}$ and
$C\subset\mathbb P^1\times\mathbb P^1$ satisfies the requirements in the definition of joining-exceptionality for $f$, then
$(M\times M)(C)$ satisfies
the same requirements for
$g$. Moreover, the definition of joining-exceptionality only requires an
invariant curve for some iterate. Thus the signed models $z^{\pm d}$ and $\pm T_d$ reduce to $z^d$ and $T_d$, as we have
$
(z^{-d})^{\circ2}=z^{d^2},
$ and for the Chebyshev maps, $-T_d$ is M\"obius-conjugate to $T_d$ when
$d$ is even, while
$
(-T_d)^{\circ2}=T_{d^2}
$
when $d$ is odd. Hence it remains only to construct the required
invariant curves for $z^d$, $T_d$, and the Latt\`es maps.

\smallskip

First suppose that $f(z)=z^d$ with $d\geq2$. Let $m,n>1$ be
coprime integers, and let $C_{m,n}\subset\mathbb P^1\times\mathbb P^1$
be the image of the map $t\mapsto (t^n,t^m)$. In affine coordinates,
$C_{m,n}$ is given by the equation $y^n=x^m$. Since $\gcd(m,n)=1$,
this parametrization is birational onto its image. Hence $C_{m,n}$ is
irreducible. Both coordinate projections are dominant, with generic
degrees $n>1$ and $m>1$, respectively. Moreover,
$(f\times f)(C_{m,n})=C_{m,n}$. Indeed, if $y^n=x^m$, then
$(y^d)^n=(x^d)^m$, so $(f\times f)(C_{m,n})\subset C_{m,n}$. Since
$f\times f$ is finite, the image $(f\times f)(C_{m,n})$ is an
irreducible algebraic curve. Hence the above inclusion is an equality.
Therefore $f$ is joining-exceptional.

\smallskip

Next suppose that $f=T_d$ is a Chebyshev map. Denote
$\eta(z):=z+z^{-1}$ so that
$\eta\circ (z\mapsto z^d)=T_d\circ\eta$. Let $C_{m,n}$ be the curve
constructed above for the monomial $z\mapsto z^d$, and define
$$
D_{m,n}:=(\eta\times\eta)(C_{m,n})\subset\mathbb P^1\times\mathbb P^1.
$$
Since $\eta\times\eta$ is finite and $C_{m,n}$ is irreducible,
$D_{m,n}$ is irreducible. The curve $D_{m,n}$ is parametrized by
$t\mapsto(\eta(t^n),\eta(t^m))=(T_n(\eta(t)),T_m(\eta(t)))$. Hence both
projections are dominant. We only need to check that the two projections
have generic degree strictly greater than one. For the first projection,
choose a non-trivial $n$-th root of unity $\zeta$. Then $t$ and
$\zeta t$ have the same first coordinate, since
$\eta((\zeta t)^n)=\eta(t^n)$. On the other hand, since
$\gcd(m,n)=1$, we have $\zeta^m\neq1$, and for generic $t$ the two
second coordinates $\eta(t^m)$ and $\eta(\zeta^m t^m)$ are distinct.
Hence the first projection is not generically one-to-one, and so it has
generic degree strictly greater than one. The same argument, exchanging
$m$ and $n$, applies to the second projection. Moreover, since
the product map
$(z\mapsto z^d)\times(z\mapsto z^d)$ preserves $C_{m,n}$ and
$\eta\circ(z\mapsto z^d)=T_d\circ\eta$, we have
$(T_d\times T_d)(D_{m,n})=D_{m,n}$. Thus $T_d$ is joining-exceptional.

\smallskip

Finally suppose $f$ is a Latt\`es map. Then there exist an elliptic
curve $E$, a finite holomorphic map $\Theta\colon E\to\mathbb P^1$ and
an affine endomorphism $L\colon E\to E$ such that
$\Theta\circ L=f\circ\Theta$. Since $\deg f\geq2$, the affine map $L$
has a fixed point; indeed, writing $L(t)=\alpha(t)+c$, the endomorphism
$\operatorname{id}-\alpha$ is non-zero and hence surjective. After
translating the origin of $E$ to such a fixed point, and replacing
$\Theta$ by the corresponding translated covering map, we may assume
that $L$ fixes the origin of $E$, so that $L$ is a group endomorphism.
Set
$$
T_\Theta:=\{a\in E:\Theta(t+a)=\Theta(t)\text{ for every }t\in E\}.
$$
This is a finite subgroup of $E$, since it acts freely by translations
on every fiber of the finite map $\Theta$. 
For $k\geq 1$, let $[k]\colon E\to E$ denote multiplication by $k$,
and let $E[k]:=\ker [k]$ be the $k$-torsion subgroup of $E$.
Since $T_\Theta$ is finite,
we can choose coprime integers $m,n>1$ such that
$E[m]\not\subset T_\Theta$ and $E[n]\not\subset T_\Theta$. Set
$M_1:=[m]$ and $M_2:=[n]$. Since multiplication maps commute with $L$,
both $M_1$ and $M_2$ commute with $L$. Set
$$
\widetilde C:=\{(M_1(t),M_2(t)):t\in E\}\subset E\times E
\qquad
\text{and}
\qquad
C:=(\Theta\times\Theta)(\widetilde C)\subset\mathbb P^1\times\mathbb P^1.
$$
The curve $C$ is irreducible as it is the image of the irreducible
curve $E$ under the finite map $\Theta\times\Theta$. Both projections
are dominant. Moreover, neither projection is generically one-to-one.
Indeed, suppose that $\pi_1|_C$ has generic degree one. Then, for generic
$t\in E$ and every $s\in E[m]$, the points $t$ and $t+s$ give the same
first coordinate on $C$, hence also the same second coordinate. Thus
$\Theta(n(t+s))=\Theta(nt)$ for generic $t$, and therefore for every
$t\in E$. Since $[n]\colon E\to E$ is surjective, it follows that
$ns\in T_\Theta$ for every $s\in E[m]$. Since $\gcd(m,n)=1$,
multiplication by $n$ permutes $E[m]$, and hence $E[m]\subset T_\Theta$,
a contradiction. Therefore $\deg(\pi_1|_C)>1$. The same argument,
exchanging $m$ and $n$, gives $\deg(\pi_2|_C)>1$. Finally, since $M_1$
and $M_2$ commute with $L$ and $L$ is surjective, we have
$(L\times L)(\widetilde C)=\widetilde C$. Applying
$\Theta\times\Theta$ and using $\Theta\circ L=f\circ\Theta$, we obtain
$(f\times f)(C)=C$. Thus $f$ is joining-exceptional.
This completes the proof.
\end{proof}

\section{Further directions and questions}
\label{sec:further-questions}
The results of this paper suggest several directions in which joining
rigidity in rational dynamics may be developed further. We first
discuss the boundary between measurable and holomorphic rigidity and
the joining-exceptional alternative, and then turn to higher-order
joinings and higher-dimensional holomorphic dynamics.

\subsection{Measurable versus holomorphic rigidity}
\label{subsec:exceptional-questions}
Theorem \ref{thm_main} naturally raises the question of recognizing
when a purely measurable joining charges the graph of a local
biholomorphism. The examples in Section \ref{sec:measurable-examples}
show why this question is genuinely analytic rather than merely a
question about the existence of measurable dependence.

Another problem is to understand more precisely the
joining-exceptional alternative.
In the diagonal case $f_1=f_2=f$, one natural problem is to characterize
joining-exceptional maps more explicitly in dynamical terms. Pakovich's
classification of diagonal invariant curves, applied to a suitable
iterate of $f$, provides a natural starting point for this question;
see \cite{PakovichInvariantCurves} and Section \ref{sec:self-joinings}.

For a general pair $(f_1,f_2)$, the corresponding problem is to
characterize those pairs for which an iterate of $f_1\times f_2$
preserves an irreducible curve whose two coordinate projections have
degree greater than one. Pakovich's work gives a structural description
of invariant curves for broad non-exceptional classes of pairs
\cite{PakovichInvariantCurves}. It would be interesting to reformulate
this description as a more direct dynamical criterion for this property
and to understand the remaining cases.

Another
question is to understand the non-discrete local analytic
families arising in case \textup{(B)}. In particular, one may ask
which additional dynamical conditions force one member of such a
family to occur on positive joining measure, and hence force the
algebraic conclusion of Theorem \ref{thm:Cnu-finite}.

\subsection{Higher-order joinings}
\label{subsec:higher-order-joinings}
Another direction is to consider higher-order joinings; see, for
example, \cite{GlasnerHostRudolph,
JanvresseDeLaRue}. For instance, a three-fold self-joining $\nu$ on
$J(f)^3$ has three pairwise projections, each of which is an ordinary
self-joining, but these pairwise projections need not determine
$\nu$.

At the purely measurable level, higher-order phenomena can already occur
without being visible in the pairwise projections: the Bernoulli shift on two symbols with equal weights
admits an ergodic non-product three-fold joining all of
whose two-coordinate projections are products. Thus a useful
higher-order theory should distinguish between such purely measurable
joinings and joinings which produce analytic or algebraic relations.

A natural geometric question is whether ergodic three-fold
self-joinings carrying positive-mass local analytic relations can be
classified in terms of invariant algebraic subvarieties of
$(\mathbb P^1)^3$.
A natural graph-type model, after passing to an iterate $f^r$, is a
subvariety of the form
$\{(x,A(x),B(x)):x\in\mathbb P^1\}$, where $A$ and $B$ commute with
an iterate of $f$. More generally, for joinings of several possibly
different rational maps, one may expect compatible systems of local or
algebraic correspondences among the coordinates.

\subsection{Higher-dimensional holomorphic dynamics}
\label{subsec:higher-dimensional-questions}
Finally, it is natural to seek an analogue of Theorem \ref{thm_main}
for holomorphic endomorphisms of $\mathbb P^k$, $k\geq2$, and their
equilibrium measures.
As in dimension one, the corresponding measure-preserving systems
are Bernoulli; see, e.g., \cite[Section 1.6]{dinh2010dynamics} and the
references therein.
One may ask whether positive joining mass on a local biholomorphic
relation forces an invariant algebraic correspondence, and whether,
in the absence of such a relation, normalized transfer maps still
produce a compact non-discrete family of holomorphic cluster maps.

The first difficulty appears already in the construction of transfer
maps. In dimension one, if $G_{1,\hat x,n}$ and $G_{2,\hat y,n}$ are
inverse branches in the two coordinates, the normalization
$$
G_{2,\hat y,n}^{-1}\circ\left(
\frac{G'_{2,\hat y,n}(0)}{G'_{1,\hat x,n}(0)}\,\operatorname{id}
\right)\circ G_{1,\hat x,n}
$$
has derivative equal to $1$ at the base point. More importantly, Koebe
distortion converts this exact derivative normalization into uniform
control on fixed domains.

In higher dimension, one can formally normalize the derivatives in a
similar way and may consider a linear normalization involving
$DG_{2,\hat y,n}(0)DG_{1,\hat x,n}(0)^{-1}$. The difficulty is that such
a normalization does not by itself give uniform geometric control. The
derivative cocycles may have several Lyapunov directions with very
different contraction rates, and the normalizing linear maps may become
highly anisotropic. Thus the one-dimensional argument based on
univalence and Koebe distortion
 does not apply directly.

Higher-dimensional inverse-branch estimates, such as those of
Berteloot--Dupont--Molino \cite{BDM} and Berteloot--Dupont \cite{BD},
provide substitutes for Koebe distortion adapted to the
Lyapunov splitting, but involve tempered coordinate changes and
small exponential errors.
To adapt the transfer-map argument to joinings, one would need to
choose compatible normalizations for the inverse branches in the two
coordinates and prove compactness on fixed reference domains.

The positive-mass local-relation case would also require a higher-dimensional
local-to-global rigidity theorem. In dimension one, the generalized
Inou theorem \cite[Theorem 2.1]{JiXieLocalRigidityJuliaSets}
used in Section \ref{sec:finite-algebraization}
promotes the local intertwining relation
between inverse branches to an invariant algebraic curve. In higher
dimension, one would need conditions under which an analogous local
relation extends to an invariant algebraic correspondence in
$\mathbb P^k\times\mathbb P^k$. Thus a higher-dimensional joining
rigidity theorem would likely require both compactness for suitably normalized
transfer maps and a higher-dimensional local-to-global rigidity result.

\printbibliography
\end{document}